\documentclass[pdflatex,sn-mathphys-num]{sn-jnl}

\usepackage{graphicx}%
\usepackage{multirow}%
\usepackage{amsmath,amssymb,amsfonts}%
\usepackage{amsthm}%
\usepackage{mathrsfs}%
\usepackage[title]{appendix}%
\usepackage{xcolor}%
\usepackage{textcomp}%
\usepackage{manyfoot}%
\usepackage{booktabs}%
\usepackage{algorithm}%
\usepackage{algorithmicx}%
\usepackage{algpseudocode}%
\usepackage{listings}%
\usepackage{cases}
\usepackage{multirow}

\theoremstyle{thmstyleone}%
\newtheorem{theorem}{Theorem}[section]
\newtheorem{proposition}{Proposition}[section]%

\theoremstyle{thmstyletwo}%
\newtheorem{remark}{Remark}[section]%

\theoremstyle{thmstyleone}%
\newtheorem{definition}{Definition}[section]%

\newtheorem{assumption}{Assumption}[section]%
\newtheorem{lemma}{Lemma}[section]%
\newtheorem{corollary}{Corollary}[section]%

\begin{document}

\title[A Riemannian QO-free method]{A  globally and superlinearly convergent QO-free method for nonlinear optimization on Riemannian  manifolds}


\author[1]{\fnm{Chunming} \sur{Tang}}

\author[2]{\fnm{Hao} \sur{He}}

\author*[3]{\fnm{Wen} \sur{Huang}}\email{wen.huang@xmu.edu.cn}

\author[4]{\fnm{Jinbao} \sur{Jian}}

\author[1]{\fnm{Ruyun} \sur{Li}}

\affil[1]{\orgdiv{School of Mathematics \& Center for Applied Mathematics of Guangxi}, \orgname{Guangxi University}, \orgaddress{\city{Nanning}, \postcode{530004}, \country{P. R. China}}}

\affil[2]{\orgdiv{School of Mathematics}, \orgname{South China University of Technology}, \orgaddress{\city{Guangzhou}, \postcode{510640}, \country{P. R. China}}}

\affil*[3]{\orgdiv{School of Mathematical Sciences}, \orgname{Xiamen University}, \orgaddress{\city{Xiamen}, \postcode{361005}, \country{P. R. China}}}

\affil[4]{\orgdiv{School of Mathematical Sciences \& Center for Applied Mathematics of Guangxi}, \orgname{Guangxi Minzu University}, \orgaddress{\city{Nanning}, \postcode{530006}, \country{P. R. China}}}


\abstract{The quadratic optimization-free (QO-free) method is a class of powerful and effective algorithms for solving nonlinearly constrained optimization problems in Euclidean spaces. The aim of the present work is to extend this method to solve optimization problems on manifolds with additional equality and inequality constraints.
We first present a specific algorithm in the manifold setting. 
At each iteration, three linear systems sharing a common linear operator are solved to determine the master search direction. In addition, a higher-order correction direction is obtained by solving a reduced linear least squares subproblem to circumvent the Maratos effect which is assumed not to arise in existing related literature.
A Riemannian arc search is then performed within the tangent space of the current iterate to generate the new iterate. Under appropriate assumptions, we establish the global and strong convergence of the proposed method. Moreover, we prove that the unit step size will eventually be accepted by the arc search, upon which the superlinear convergence of the algorithm is established.
Finally, numerical results demonstrate that the proposed method is very competitive compared with other existing approaches.}

\keywords{Constrained optimization, Riemannian manifolds, QO-free method,  Global convergence, Superlinear convergence.}



\maketitle

\section{Introduction}\label{sec1}
We consider the following problem
\begin{equation}\label{Rpro}
	\begin{aligned}
		&\min_{x\in\mathcal{M}}\ \ f(x)\\
		&{\rm \,\,s.t.}\,\,\,\ \ c_i(x)\leq 0,\quad i\in\mathcal{I}=\{1,\cdots,m\},\\
		&\quad\,\,\,\,\,\,\,\ \ c_i(x)= 0,\quad i\in\mathcal{E}=\{m+1,\cdots,m+\ell\},
	\end{aligned}\tag{P}
\end{equation}
where $\mathcal{M}$ is a complete Riemannian manifold with dimension $n$, and $f:\mathcal{M}\rightarrow \mathbb{R}$ and $c_i:\mathcal{M}\rightarrow \mathbb{R}$,  $i\in\mathcal{I}\cup\mathcal{E}$, are twice continuously differentiable functions.
Problems of this kind frequently arise in practical applications, such as 
nonnegative principal component analysis \cite{montanari2015non},
multicontact postures computation of static robot \cite{brossette2018multicontact},
direct trajectory optimization of rigid bodies \cite{teng2025riemannian}, 
$k$-means clustering \cite{carson2017manifold}, 
and robust Wasserstein distance \cite{jiang2024riemannian}.

If $\mathcal{M}=\mathbb{R}^n$, then problem \eqref{Rpro} reduces to the standard 
nonlinear constrained optimization problem in the Euclidean space $\mathbb{R}^n$.
For this case, there are already numerous efficient solution methods available, such as penalty function methods \cite{bertsekas2016nonlinear},
sequential quadratic optimization (SQO) or formerly known as sequential quadratic programming (SQP) methods \cite{gill2010sequential,lawrence2001computationally,jian2022sequential,jian2025partially},  quadratic optimization-free (QO-free) methods \cite{panier1988qp,gao1997sequential,kanzow1999qp,qi2000new,yang2003feasible,chen2006feasible,jian2010new,qi2002globally,facchinei2003local,liu2014infeasible}, and interior point methods \cite{nocedal2006numerical}. 
On the other hand, if there are no additional constraints in problem \eqref{Rpro}, i.e.,  
$\mathcal{I}=\emptyset$ and $\mathcal{E}=\emptyset$, then \eqref{Rpro} reduces to the classical optimization on Riemannian manifolds (or called Riemannian optimization). In this context, a variety of efficient methods have also been developed, such as Riemannian steepest descent method \cite{munier2006steepest}, Riemannian conjugate gradient methods \cite{sato2015new,zhu2017riemannian,sakai2021sufficient,Tang2023}, Riemannian Newton-type methods \cite{RBFGS02,HuangBroyden,Youse2016}, Riemannian trust-region methods \cite{RTR07,Grohs2016}. We refer the readers to the books \cite{absil2008optimization,Boumal2023,sato2021} and the survey \cite{hu2020brief} for more details on Riemannian optimization.

In contrast to the extensive research achievements gained for the two aforementioned special cases, studies on the general problem \eqref{Rpro} remain quite limited.
Below, we present a brief review of the relevant literature. 
In 2014, Yang et al. \cite{yang2014optimality} generalized the linear independent constraint qualification (LICQ), and the optimality conditions including the Karush-Kuhn-Tucker (KKT) conditions, the second-order necessary conditions, and the second-order sufficient conditions (SOSC) to the Riemannian setting. Later, Bergmann and Herzog~\cite{bergmann2019intrinsic} extended some other constraint qualifications (CQ) from Euclidean spaces to the manifold setting, including 
Mangasarian-Fromovitz CQ, Abadie CQ and Guignard CQ. 
In 2018, Brossette et al.~\cite{brossette2018multicontact} proposed a filter-type Riemannian SQO (RSQO) method to solve problem \eqref{Rpro} and applied it to the computation of multicontact postures for static robots; however, no convergence analysis was provided.
In 2020, Liu and Boumal \cite{liu2020simple} extended an augmented Lagrangian method and an exact penalty method---originally developed for the Euclidean setting---to the Riemannian setting, and further established their global convergence. Yet, the work \cite{liu2020simple} did not include a local convergence analysis.
In 2021, Schiela and Ortiz \cite{schiela2021sqp} developed an RSQO method for equality constrained problems (i.e., $\mathcal{I} = \emptyset$) on Hilbert manifolds, and established its local quadratic convergence.
In 2022, Obara et al. \cite{obara2022sequential} proposed a penalty-based RSQO method to solve problem \eqref{Rpro}, which reduces to the classical Euclidean Han's SQO method \cite{han1977globally} when $\mathcal{M}=\mathbb{R}^n$. Note that this is the first algorithm that enjoys both global and local convergence for \eqref{Rpro}. 
We also note that while the method in \cite{obara2022sequential} exhibits excellent numerical performance on two classes of test problems, it suffers from two theoretical limitations in general: 1) the quadratic optimization (QO) subproblem solved at each iteration may become infeasible (i.e., inconsistent) when the iteration point violates the equality and/or inequality constraints, thereby causing the algorithm to fail; 2) the Maratos effect \cite{maratos1978exact} may occur, meaning that close to a solution, the merit function may fail to accept the unit step. 
In fact, overcoming the Maratos effect is a necessary prerequisite for proving the local fast convergence of the algorithm, yet the method proposed in \cite{obara2022sequential} merely takes this property as an underlying assumption. Recently, Obara et al. \cite{obara2024stable} applied the method proposed in \cite{obara2022sequential} to the identification of stable linear systems.
In 2024, Lai and Yoshise \cite{lai2024riemannian} generalized the classical primal-dual interior-point method to the Riemannian setting, derived its local superlinear and quadratic convergence properties within the framework of the perturbed damped Riemannian Newton method, and further established the global convergence of the algorithm when combined with a classical line search. However, the Maratos effect may still be unavoidable in this context. Very recently, Obara et al. \cite{obara2025local} (and its companion paper \cite{obara2025primal}) proposed a Riemannian interior point method for problem \eqref{Rpro}, and established the local superlinear convergence and local near-quadratic convergence. Furthermore, when $\mathcal{E}=\emptyset$, they also demonstrated the global convergence by using a trust region approach.

From the above observations, we note that the existing methods for solving problem \eqref{Rpro} either lack convergence analysis \cite{brossette2018multicontact}, only achieve global convergence \cite{liu2020simple}, only exhibit local convergence \cite{schiela2021sqp}, or possess both global and local convergence but fail to overcome the Maratos effect \cite{obara2022sequential,lai2024riemannian}.
The goal of this paper is to propose a new method for problem \eqref{Rpro} that enjoys the following properties:
\begin{itemize}
	\item  the main subproblems solved at each iteration are always feasible;
	\item  both global and local fast convergence can be achieved simultaneously;
	\item  the Maratos effect can be overcome.
\end{itemize}

Some techniques in this paper originate from the works in~\cite{qi2000new},~\cite{mayne1976feasible},~\cite{tits2003primal}, and~\cite{panier1988qp}. 
In \cite{qi2000new}, a Euclidean QO-free method was proposed. Unlike SQO methods, QO-free methods avoid the need to solve QO subproblems; instead, they only require solving systems of linear equations. 
Over the past few decades, Euclidean QO-free methods have been well-studied and have evolved into diverse forms; see, e.g., \cite{panier1988qp,gao1997sequential,kanzow1999qp,qi2000new,yang2003feasible,chen2006feasible,jian2010new,qi2002globally,facchinei2003local,liu2014infeasible}. The framework of the proposed Riemannian QO-free method mainly follows that in~\cite{qi2000new}. 
Since the Euclidean QO-free method in \cite{qi2000new} only addresses problems with inequality constraints, our approach adopts the semi-penalty scheme proposed by Mayne and Polak \cite{mayne1976feasible} to convert problem \eqref{Rpro} into a sequence of inequality-constrained optimization problems of the following form:
\begin{equation}\label{TRpro}
	\begin{aligned}
		&\min_{x\in\mathcal{M}}&  F_\rho(x):=f(x)-\rho\sum_{i\in\mathcal{E}}c_i(x)\\
		&{\rm \,\,s.t.}& c_i(x)\leq 0,\quad i\in \mathcal{L}=\mathcal{I}\cup\mathcal{E},
	\end{aligned}\tag{${\rm P}_\rho$}
\end{equation}
where $\rho>0$ is a penalty parameter. Intuitively, large values of $\rho$ will force the iterates to satisfy the equality constraints $c_i(x)=0, i\in \mathcal{E}$. 
In our algorithm, we employ the scheme proposed in \cite{tits2003primal} to increase 
$\rho$, ensuring that it remains fixed after a finite number of iterations. Moreover, we also establish the relationship between problems \eqref{Rpro} and \eqref{TRpro}; see Proposition \ref{tp-p} below.
At each iteration, three linear systems sharing a common linear operator are solved to determine the master search direction, which yields significant computational advantages compared to using different linear operators. For instance, one can use LU decomposition to solve the linear systems; consequently, the computational cost is essentially equivalent to that of solving a single linear system.
These systems are always feasible and admit unique solutions.
The first one is solved to generate a descent direction for $F_\rho$,
the second one is used to generate a sufficient descent direction,
and the third one is to ensure the feasibility of the next iterate.
As a result, a master direction with feasible and descent properties is obtained.
In order to further overcome the Maratos effect, a higher-order correction direction is produced by solving a reduced linear least squares subproblem.
Note that this subproblem does not need to be solved in early iterations.
A Riemannian arc search is then performed within the tangent space of the current iterate to generate the new iterate. 
Our method belongs to the class of feasible methods for problem \eqref{TRpro}; that is, it starts from a (strictly) feasible point and generates a sequence of (strictly) feasible iterates for \eqref{TRpro}. It further qualifies as a feasible descent method once the parameter $\rho$ is fixed. 

Under appropriate assumptions, we establish the global convergence of the proposed method. It is worth mentioning that this global convergence can also be achieved even in the absence of the higher-order correction direction. The strong convergence (i.e., the convergence of the whole iteration sequence) is also achieved.
Moreover, we prove that the unit step size will eventually be accepted by the arc search, upon which the superlinear convergence of the algorithm is established.
Finally, numerical results demonstrate that the proposed method is competitive compared with other existing approaches. As far as we know, our method is the first extension of QO-free methods to the Riemannian setting. Additionally, this is the first time that the Maratos effect has been overcome in the context of solving problem \eqref{Rpro}.

The organization of this paper is as follows. In Section \ref{Sec:Preliminaries}, 
we review some basic results closely relevant to our method.
In Section \ref{Sec:algorithm}, we present the algorithm and discuss its properties.
Global convergence and superlinear convergence are established in Sections \ref{Sec:global} and \ref{Sec:superlinear}, respectively. Numerical experiments are provided in Section \ref{Sec:numerical}.
Conclusions are presented in Section \ref{Sec:Conclusion}.

\section{Preliminaries}\label{Sec:Preliminaries}

In this section, we present a review of some fundamental results for Riemannian optimization, as well as the method introduced in \cite{qi2000new}.

\subsection{Notations and Riemannian optimization}
Our notation follows the standard conventions described in \cite{absil2008optimization}.
The tangent space and the tangent bundle of $\mathcal{M}$ are denoted by $T_x\mathcal{M}$ and $T\mathcal{M}$, respectively. 
The inner product on $T_{x}\mathcal{M}$ is denoted by $\left \langle \cdot,\cdot \right \rangle_{x}$, and	$\|\xi\|_x:=\sqrt{\left \langle \xi, \xi \right \rangle_{x}}$ denotes the norm of a tangent vector $\xi\in T_{x}\mathcal{M}$.
We usually omit the subscript $x$ if it is clear from the context. Given two points $x,y\in\mathcal{M}$, the Riemannian distance is denoted by ${\rm dist}(x,y)$. For a linear operator $\mathcal{A}:\mathbb{W}\to\mathbb{V}$, we denote the spectral norm (or operator norm) of $\mathcal{A}$ by $\|\mathcal{A}\|_{\rm op}$, with $\mathbb{W}$, $\mathbb{V}$ being normed linear spaces.

It is well known that the concept of ``gradient" plays a crucial role in the design of optimization algorithms. Thus, we first introduce the definition of the gradient of a function on a Riemannian manifold, which is referred to as the Riemannian gradient.
\begin{definition}
	(\cite[Sec. 3.6]{absil2008optimization}) Let $f:\mathcal{M}\rightarrow \mathbb{R}$ be a smooth function on $\mathcal{M}$. The gradient of $f$ at $x\in \mathcal{M}$, denoted by ${\rm grad}f(x)$, is defined as the unique element of $T_x\mathcal{M}$ that satisfies 
	$$\langle {\rm grad}f(x), \xi\rangle_x={\rm D}f(x)[\xi],\ \forall \xi\in T_x\mathcal{M},$$
	where ${\rm D}f(x)$ is the differential of $f$ at $x$.
\end{definition}


Let $\bar{\nabla}$ denote the Riemannian connection, which is used to characterize the derivative of a vector field on $\mathcal{M}$; see \cite[Sec. 5]{absil2008optimization} for further details.
Based on this, we introduce the concept of Riemannian Hessian that enables us to establish the SOSC for \eqref{Rpro}.

\begin{definition}
	(\cite[Def. 5.5.1]{absil2008optimization}) Given a smooth function $f:\mathcal{M}\rightarrow \mathbb{R}$, the Riemannian Hessian of $f$ at $x\in\mathcal{M}$ is a linear mapping ${\rm Hess} f(x):T_x\mathcal{M}\rightarrow T_x\mathcal{M}$  defined by
	$${\rm Hess}f(x)[\xi_x]=\bar{\nabla}_{\xi_x}{\rm grad}f,\quad \forall \xi_x\in T_x\mathcal{M},$$
	where $\bar{\nabla}_{\xi_x}{\rm grad}f$ denotes the derivative of the vector field $\mathrm{grad} f$ along direction $\xi_x$.
\end{definition}


The following notion of retraction is used to define the update of iterates in Riemannian optimization algorithms.
\begin{definition}
	(\cite[Def. 4.1.1]{absil2008optimization}) A smooth mapping $R:T\mathcal{M}\rightarrow \mathcal{M}$ is called a retraction on $\mathcal{M}$ if it satisfies the properties: (i) $R_x(0_x)=x$, (ii) ${\rm D}R_x(0_x)={\rm id}_x,$ where $0_x$ denotes the zero element of $T_x\mathcal{M}$, $R_x$ and ${\rm id}_x$ denote the restriction of $R$ to $T_x \mathcal{M}$ and the identity mapping on $T_x\mathcal{M}$, respectively.
\end{definition}



Denote the feasible set and the Lagrangian function of problem \eqref{Rpro} by, respectively,
$$\mathcal{F}_P=\{x\in\mathcal{M}\mid c_i(x)\leq 0,\,i\in\mathcal{I};\,c_i(x)= 0,\,\, i\in\mathcal{E}\},$$
$$L(x,\lambda)=f(x)+\sum_{i\in\mathcal{L}}\lambda_ic_i(x),$$ 
where $\mathcal{L}=\mathcal{I}\cup\mathcal{E}$, and $\lambda\in\mathbb{R}^{m+\ell}$ is referred to as the Lagrangian multiplier vector. 
It is said that the linear independence constraint qualification (LICQ) holds on $\mathcal{M}$ at $x\in\mathcal{F}_P$ if
\begin{equation}\label{LICQ}
	\{{\rm grad}c_i(x)\mid i\in \mathcal{I}(x)\cup\mathcal{E}\}{\rm \,\,is\,\, linearly\,\, independent\,\, on}\,\,T_x\mathcal{M},
\end{equation}
where $\mathcal{I}(x)=\{i\in \mathcal{I} \mid c_i(x)=0\}$; see \cite{yang2014optimality}.

The conditions presented below are extensions of the Euclidean KKT conditions to the Riemannian setting.
\begin{definition}
	(\cite[Sec. 4]{yang2014optimality}) The point $x^*\in\mathcal{F}_P$ is said to be a KKT point for problem \eqref{Rpro} if there exists a vector $\lambda^*\in\mathbb{R}^{m+\ell}$ such that
	\begin{subequations}\label{kkt}
		\begin{align}
			\label{kkt1}
			&{\rm grad}f(x^*)+\sum_{i\in\mathcal{I}\cup\mathcal{E}}\lambda^*_i{\rm grad}c_i(x^*)=0_{x^*},\\
			\label{kkt2}
			&\lambda^*_i\geq 0,\,\,c_i(x^*)\leq 0,\quad \forall i\in\mathcal{I},\\
			\label{kkt3}
			&\lambda^*_ic_i(x^*)=0,\quad \forall i\in \mathcal{I},\\
			&c_i(x^*)=0,\quad\forall i\in\mathcal{E}.
		\end{align}
	\end{subequations}
	Further, $\lambda^*$ is referred to as the KKT multiplier vector, and the pair 
	$(x^*,\lambda^*)$ is called the KKT pair of problem \eqref{Rpro}.
\end{definition}


In \cite[Thm. 4.1]{yang2014optimality}, it is shown that if $x^*$ is a local solution of problem \eqref{Rpro} and the LICQ \eqref{LICQ} holds at $x^*$, then $x^*$ is a KKT point of  \eqref{Rpro}. The goal of our paper is to find KKT points of \eqref{Rpro}.  Furthermore, to establish superlinear convergence of the proposed method, we require the following condition.

\begin{definition}\label{Def5}
	(\cite[Sec. 4]{yang2014optimality}) Let $(x^*,\lambda^*)$ be a KKT pair of problem \eqref{Rpro}. 
	We say that the SOSC hold at $(x^*,\lambda^*)$ if 
	\begin{equation*}
		\left\langle{\rm Hess}_xL(x^*,\lambda^*)[\xi],\xi\right\rangle_{x^*}>0,\quad \forall \xi\in \mathcal{C}(x^*,\lambda^*),\,\,\xi\neq 0_{x^*},
	\end{equation*}
	where  ${\rm Hess}_xL(x^*,\lambda^*)$ denotes the Riemannian Hessian of the Lagrangian  function $L(x,\lambda)$ with respect to $x$, and
	\begin{equation*}
		\mathcal{C}(x^*,\lambda^*)=\left\{\xi_{x^*}\in T_{x^*}\mathcal{M} \,\,\Bigg\arrowvert\,\,\begin{aligned}
			&\langle\xi_{x^*},{\rm grad}c_i(x^*) \rangle=0, \quad \forall i\in \mathcal{I}(x^*)\,\,{\rm with}\,\,\lambda^*_i>0 \\
			&\langle\xi_{x^*},{\rm grad}c_i(x^*) \rangle\leq 0, \quad \forall i\in \mathcal{I}(x^*)\,\,{\rm with}\,\,\lambda^*_i=0 \\
			&\langle\xi_{x^*},{\rm grad}c_i(x^*) \rangle=0, \quad \forall i\in \mathcal{E} 
		\end{aligned}
		\right\}.
	\end{equation*}
\end{definition}

\subsection{Brief review of the Euclidean QO-free method in \cite{qi2000new}}

Consider problem \eqref{Rpro} with $\mathcal{M}=\mathbb{R}^n$ and $\mathcal{E}=\emptyset$, i.e.,
\begin{equation}\label{QOfree_pro}
	\begin{aligned}
		\min_{x\in \mathbb{R}^n}\,\,&f(x)\\
		{\rm s.t.}\,\,\,\,&c_i(x)\leq 0,\,\,\,\,i\in\mathcal{I}.
	\end{aligned}
\end{equation}
A pair $(x, \lambda)\in\mathbb{R}^{n+m}$ is called a KKT pair of \eqref{QOfree_pro} if it satisfies
\begin{equation}\label{QOfre_kkt}
	\nabla_x L(x,\lambda)=0,\ c_i(x)\leq 0,\  \lambda_i\geq 0,\      \lambda_ic_i(x)=0,\,\,\,\,i\in\mathcal{I},
\end{equation}
where $L(x,\lambda)=f(x)+\sum_{i\in\mathcal{I}}\lambda_ic_i(x)$ is the Lagrangian function of \eqref{QOfree_pro}. 
By making use of the Fischer Burmeister nonlinear complementarity problem function
\begin{equation*}
	\phi(a,b)=\sqrt{a^2+b^2}-a-b,\,\,\ \ \ a,b\in\mathbb{R},
\end{equation*}
the KKT conditions of \eqref{QOfree_pro} can be equivalently transformed as
\begin{equation}\label{Phixlam}
	\Phi(x,\lambda)=\begin{pmatrix}
		\nabla_x L(x,\lambda)\\
		\phi(-c_1(x),\lambda_1)\\
		\vdots\\
		\phi(-c_m(x),\lambda_m)
	\end{pmatrix}=0.
\end{equation}
Consider a Newton-type method for solving equation \eqref{Phixlam}. At the current iterate $(x^k,\mu^k)$, the Jacobian of $\Phi$ is given by 
\begin{equation}
	\Phi'(x^k,\mu^k)=\begin{pmatrix}
		\nabla_x^2 L(x^k,\mu^k)&\nabla c(x^k)\\
		{\rm diag}(\alpha^k)\nabla c(x^k)^{\top} & -({\rm diag}(\beta^k))^2
	\end{pmatrix},
\end{equation}
where ${\rm diag}(v)$ denotes the diagonal matrix whose $i$th diagonal element is $v_i$; $\nabla c(x^k)=(\nabla c_1(x^k),\cdots,\nabla c_m(x^k))$; $\nabla c_i$ denotes the Euclidean gradient of the  $c_i$; and $\alpha^k,\beta^k\in\mathbb{R}^m$, whose $i$th elements are defined as
\begin{equation}\label{alpha}
	\alpha_i^k=\dfrac{c_i(x^k)}{\sqrt{c_i^2(x^k)+(\mu_i^k)^2}}+1,\quad\beta_i^k=\left(1-\dfrac{\mu_i^k}{\sqrt{c_i^2(x^k)+(\mu_i^k)^2}}\right)^{1/2}.
\end{equation}
In order to apply quasi-Newton methods to solve \eqref{QOfree_pro} and achieve superlinear convergence, Qi and Qi \cite{qi2000new} replaced the Jacobian matrix $\Phi'(x^k,\mu^k)$ with
\begin{equation}\label{V_k}
	V_k:=\begin{pmatrix}
		H_k & \nabla c(x^k)\\
		{\rm diag}(\alpha^k)\nabla c(x^k)^{\top} & -\sqrt{2}{\rm diag}(\beta^k)
	\end{pmatrix},
\end{equation}
where $H_k$ is a symmetric positive definite matrix approximating $\nabla_x^2 L(x^k,\mu^k)$. 
At each iteration, the QO-free method in \cite{qi2000new} produces the master direction $d^k$ by solving three systems of linear equations of the form 
\begin{equation}\label{coefficient}
	V_k
	\begin{pmatrix}
		d\\
		\lambda
	\end{pmatrix}
	=
	\begin{pmatrix}
		-\nabla f(x^k)\\
		v^k
	\end{pmatrix},
\end{equation}
where $v^k$ is an appropriately selected vector for each system, with the goal of guaranteeing the descent and feasibility of the direction.

Furthermore, to avoid the Maratos effect, the higher-order correction direction $\tilde{d}^k$ is obtained by solving the following least squares subproblem
\begin{equation*}
	\begin{aligned}
		&\min_{d\in \mathbb{R}^n}\ \ \dfrac{1}{2}\left\langle d,H_kd\right\rangle\\
		&{\rm \,\,s.t.}\,\,\, c_i(x^k+d^k)+\langle\nabla c_i(x^k),d\rangle=-\varpi_k,\quad \forall i\in \mathcal{I}_k,
	\end{aligned}
\end{equation*}
where $\varpi_k>0$ and $\mathcal{I}_k\subset \mathcal{I}$.
An arc search is then performed to generate a step size $t_k$, and the new iterate is obtained as  
$x^{k+1}=x^k+t_kd^k+t_k^2\tilde{d}^k$.

\section{Riemannian QO-free Algorithm}\label{Sec:algorithm}

The proposed Riemannian QO-free algorithm is described in Section~\ref{sec:02} and the well-definiteness of this algorithm is shown in Section~\ref{sec:03}.
More precisely, we first investigate the relationship between the original problem \eqref{Rpro} and the penalized problem \eqref{TRpro}. We then elaborate on the proposed method and its corresponding algorithm. Finally, we prove the well-definedness of the proposed algorithm to conclude this section.

\subsection{Description of the algorithm}~\label{sec:02}

The feasible set, strictly feasible set and the Lagrangian function of problem \eqref{TRpro} are denoted by
$\mathcal{F}=\{x\in\mathcal{M}\mid c_i(x)\leq 0, i\in \mathcal{L}\}$, ${\rm int}\mathcal{F}=\{x\in \mathcal{M}\mid c_i(x)< 0, i\in \mathcal{L}\}$, and $L_\rho(x,\lambda)=F_\rho(x)+\sum_{i\in \mathcal{L}}\lambda_ic_i(x)$, respectively. 
In what follows, a point $x$ is said to be strictly feasible for \eqref{TRpro} if $x\in {\rm int}\mathcal{F}$.
Additionally, let $\lambda_\mathcal{S}=\{\lambda_i, i\in \mathcal{S}\}$ for some set $\mathcal{S} \subseteq\mathcal{L}$, and  $e=(1,\dots,1)^{\top}$ with a suitable dimension.


Proposition~\ref{tp-p} clarifies the relationship between \eqref{Rpro} and \eqref{TRpro}; its proof can be found in Appendix~\ref{AppendixA}.
\begin{proposition}\label{tp-p} 
	Given a penalty parameter $\rho>0$. 
	\begin{description}
		\item[(i)] If $(x^*,\lambda^*)$ is a KKT pair for \eqref{TRpro} with  $c_i(x^*)=0$ for all $i\in\mathcal{E}$, then $(x^*,\lambda_\mathcal{I}^*,\lambda_\mathcal{E}^*-\rho e)$ is a KKT pair for \eqref{Rpro}.
		\item[(ii)] If $\lambda_{\mathcal{E}^*}\geq0$ and $(x^*,\lambda_\mathcal{I}^*,\lambda_\mathcal{E}^*-\rho e)$ is a KKT pair of problem \eqref{Rpro}, then $(x^*,\lambda^*)$ is a KKT pair of problem \eqref{TRpro}.
	\end{description}
\end{proposition}

The condition ``$c_i(x^*)=0$ for all $i\in\mathcal{E}$" in proposition \ref{tp-p} can be guaranteed by a sufficiently large penalty parameter $\rho$ (see the proof of Theorem \ref{th2}), which inspires us to increase $\rho$ during iterations. 

Our algorithm starts from a strict feasible point $x^0\in {\rm int}\mathcal{F}$, and generates a sequence of strict feasible iterates of problem \eqref{TRpro}.
For the $k$-th iteration $x^k\in{\rm int}\mathcal{F}$, we first obtain $(\eta^{k0},\lambda^{k0})\in T_{x^k}\mathcal{M}\times\mathbb{R}^{m+\ell}$  by solving the following linear system:
\begin{subnumcases}
	{\label{s1}}
	\label{s11}
	\mathcal{H}_k[\eta]+\sum_{i\in \mathcal{L}}\lambda_i{\rm grad}c_i(x^k)=-{\rm grad}F_{\rho_k}(x^k),\\
	\label{s12}
	\alpha_i^k\langle{\rm grad}c_i(x^k),\eta\rangle-\sqrt{2}\beta_i^k\lambda_i=0,\quad \forall i\in \mathcal{L},	
\end{subnumcases}
where $\mathcal{H}_k:T_{x^k}\mathcal{M}\rightarrow T_{x^k}\mathcal{M}$ is a linear and symmetric positive definite operator.
The scalars $\alpha_i^k, \beta_i^k, i\in \mathcal{L}$ are of the form given in \eqref{alpha}.
These are well-defined because $x^k\in{\rm int}\mathcal{F}$, i.e. $c_i(x^k)<0$ for all $i\in \mathcal{L}$. It has been shown in~\cite{fischer1992special,fischer1995ncp} that $\alpha_k^k$ and $\beta_i^k$ satisfy
\begin{equation}\label{alha-beta}
	\alpha_i^k\in[0,1),\,\,\beta_i^k\in(0,\sqrt{2})\,\,{\rm and}\,\, (\alpha_i^k)^2+(\beta_i^k)^4\geq 3-2\sqrt{2},\quad \forall i\in \mathcal{L}.
\end{equation}

The direction $\eta^{k0}$ generated by \eqref{s1} is a descent direction for the function $F_{\rho_k}(x)$ at $x^k$. To the end, by \eqref{s11}, we have
\begin{equation}\label{dk0}
	\langle{\rm grad}F_{\rho_k}(x^k),\eta^{k0}\rangle
	=  -\left\langle \eta^{k0},\mathcal{H}_k[\eta^{k0}]\right\rangle-\sum_{i\in \mathcal{L}}\lambda_i^{k0}\left\langle{\rm grad}c_i(x^k),\eta^{k0}\right\rangle.
\end{equation}
If $\mu_i^k=0$, we must have $\alpha_i^k=0$ and $\beta_i^k=1$ by \eqref{alpha}. It follows from \eqref{s12} that $\lambda_i^{k0}=0$, which yields $\langle{\rm grad}F_{\rho_k}(x^k),\eta^{k0}\rangle
=  -\left\langle \eta^{k0},\mathcal{H}_k[\eta^{k0}]\right\rangle$. If $\mu_i^k\neq 0$, then $\alpha_i^k\neq 0$. Thus, \eqref{s12} and \eqref{dk0} yield that 
\begin{equation*}
	\begin{aligned}
		\langle{\rm grad}F_{\rho_k}(x^k),\eta^{k0}\rangle
		&=  -\left\langle \eta^{k0},\mathcal{H}_k[\eta^{k0}]\right\rangle-\sum_{i\in J^k_\mu}\dfrac{\sqrt{2}\beta_i^k(\lambda_i^{k0})^2}{\alpha_i^k} \\\
		& \leq -\left\langle \eta^{k0},\mathcal{H}_k[\eta^{k0}]\right\rangle,
	\end{aligned}
\end{equation*}
where $J^k_\mu=\{i\in \mathcal{L}\mid \mu_i^k\neq 0\}$. This together with the positive definiteness of $\mathcal{H}_k$ implies  
\begin{equation}\label{dk01}
	\langle{\rm grad}F_{\rho_k}(x^k),\eta^{k0}\rangle\leq -\left\langle \eta^{k0},\mathcal{H}_k[\eta^{k0}]\right\rangle <0.
\end{equation}


To find a sufficient descent direction, we solve the following system---a perturbation of \eqref{s1}---to obtain $(\eta^{k1},\lambda^{k1})\in T_{x^k}\mathcal{M}\times\mathbb{R}^{m+\ell}$:
\begin{subnumcases}
	{\label{s2}}
	\label{s21}
	\mathcal{H}_k[\eta]+\sum_{i\in \mathcal{L}}\lambda_i{\rm grad}c_i(x^k)=-{\rm grad}F_{\rho_k}(x^k),\\
	\label{s22}
	\alpha_i^k\langle{\rm grad}c_i(x^k),\eta\rangle-\sqrt{2}\beta_i^k\lambda_i=\alpha_i^k(\min\{\lambda_i^{k0},0\})^3,\quad \forall i\in \mathcal{L}.
\end{subnumcases}
In Lemma \ref{le3} below, it is shown that the descent amount of $\eta^{k1}$ is greater than that of $\eta^{k0}$.  However, $\eta^{k1}$ alone does not ensure convergence. To derive the theoretical results, we next solve the following linear system to compute  $(\eta^{k2},\lambda^{k2})\in T_{x^k}\mathcal{M}\times\mathbb{R}^{m+\ell}$:
\begin{subnumcases}
	{\label{s3}}
	\label{s31}
	\mathcal{H}_k[\eta]+\sum_{i\in \mathcal{L}}\lambda_i{\rm grad}c_i(x^k)=-{\rm grad}F_{\rho_k}(x^k),\\
	\label{s32}
	\alpha_i^k\langle{\rm grad}c_i(x^k),\eta\rangle-\sqrt{2}\beta_i^k\lambda_i=\alpha_i^k(\min\{\lambda_i^{k0},0\})^3-\alpha_i^k\|\eta^{k1}\|^\nu,\quad \forall i\in \mathcal{L},
\end{subnumcases}
where $\nu\in(2,+\infty)$ is a constant. Note that \eqref{s3} is a perturbation of \eqref{s2}, where the term ``$-\alpha_i^k\|\eta^{k1}\|^\nu$" is added to enhance the feasibility of the resulting direction.

Subsequently, the master search direction $\eta^k$ is defined as a linear combination of $\eta^{k1}$ and $\eta^{k2}$ as follows
\begin{equation}\label{direction}
	\eta^k=(1-\theta_k)\eta^{k1}+\theta_k\eta^{k2},\,\,\lambda^k=(1-\theta_k)\lambda^{k1}+\theta_k\lambda^{k2},
\end{equation}
where the parameter $\theta_k$ plays a crucial role in balancing the descent property and feasibility. 
Rather than extending the Euclidean form proposed in \cite{qi2000new}, we extend the one from \cite{panier1988qp} to the following form: 
\begin{eqnarray}\label{thetak}
	\theta_k=
	\left \{
	\begin{array}{ll}
		1,& {\rm if}\ \langle{\rm grad}F_{\rho_k}(x^k),\eta^{k2}\rangle\leq \tau\langle{\rm grad}F_{\rho_k}(x^k),\eta^{k1}\rangle,\\
		\dfrac{(1-\tau)\langle{\rm grad}F_{\rho_k}(x^k),\eta^{k1}\rangle}{\langle{\rm grad}F_{\rho_k}(x^k),\eta^{k1}-\eta^{k2}\rangle},& {\rm otherwise},
	\end{array}
	\right.
\end{eqnarray}
where $\tau\in(0,1).$


Finally, to avoid the Maratos effect, we solve the following linear least squares subproblem to obtain a higher-order correction direction $\tilde{\eta}^k\in T_{x^k}\mathcal{M}$:
\begin{equation}\label{subpro}
	\begin{aligned}
		&\min_{\eta\in T_{x^k}\mathcal{M}} \ \ \dfrac{1}{2}\left\langle \eta,\mathcal{H}_k[\eta]\right\rangle\\
		&{\rm \,\,s.t.}\,\,\, \ c_i(R_{x^k}(\eta^k))+\langle{\rm grad}c_i(x^k),\eta\rangle=-\varpi_k,\quad \forall i\in \mathcal{L}_k,
	\end{aligned}
\end{equation}
where $\mathcal{L}_k=\{i\in \mathcal{L}\mid c_i(x^k)\geq -\lambda_i^k\}$, and
\begin{equation*}
	\varpi_k=\max\left\{\|\eta^k\|^\varrho,\,\,\max_{i\in \mathcal{L}_k}\bigg\arrowvert\dfrac{\alpha_i^k}{\sqrt{2}\delta_i^k\lambda_i^k}-1\bigg\arrowvert^\kappa\|\eta^k\|^2\right\}
\end{equation*}
with $\varrho\in(2,3)$, $\kappa\in(0,1)$, and $\delta_i^k=-\beta_i^k/c_i(x^k)$, $i\in \mathcal{L}_k$. We point out that subproblem \eqref{subpro} is an extension of that in \cite{qi2000new}, and it is not required to solve for global convergence.


\vskip 0.3cm
Building on the foregoing analysis, we formally present our algorithm in Algorithm~\ref{algo1}.
\begin{algorithm}
	\caption{A Riemannian QO-free method (RQO-free)}\label{algo1}
	\begin{algorithmic}[1]
		\Require Initial point $x^0 \in {\rm int}\mathcal{F}$; initial Hessian approximation $\mathcal{H}_0:T_{x^0}\mathcal{M}\rightarrow T_{x^0}\mathcal{M}$ which is a symmetric positive definite linear operator; parameters: $\nu\in(2,+\infty)$, $\tau\in(0,1)$,  $\varrho\in(2,3)$,  $\kappa\in(0,1)$, $\sigma\in(0,1/2)$,  $\varsigma\in(0,1)$, $\tilde{\rho}>1$, $\rho_0>0$, $r_1>0$, $r_2>0$, $r_3>0$,  $\bar{\mu}>0$, $\mu_i^0\in(0,\bar{\mu}]$, $i\in \mathcal{L}$.
		\For{$k=0,1,\cdots$}
		\State Compute $(\eta^{k0},\lambda^{k0})$ by solving  \eqref{s1}.
		\If {(a) $\|\eta^{k0}\|\leq r_1$, (b)  $\lambda^{k0}_\mathcal{E}\ngeq r_2e$, and (c) $\lambda^{k0}\geq -r_3e$ hold} 
		
		\State Set $\rho_{k+1}=\tilde{\rho}\rho_k$, $x^{k+1}=x^k$, $\mathcal{H}_{k+1}=\mathcal{H}_k$, $\mu^{k+1}=\mu^k$.
		
		\State	\textbf{continue}.
		\EndIf
		
		\State Compute $(\eta^{k1},\lambda^{k1})$ by solving \eqref{s2}.
		\State Compute $(\eta^{k2},\lambda^{k2})$ by solving \eqref{s3}.
		\State Obtain $\eta^k$ and $\lambda^k$ according to \eqref{direction}.
		\State Try to find $\tilde{\eta}^k$ by solving the subproblem ($\ref{subpro}$).
		\If{$\|\tilde{\eta}^k\|>\|\eta^k\|$ or \eqref{subpro} has no solution}
		\State Set $\tilde{\eta}^k=0_{x^k}$.
		\EndIf
		
		\State Compute $t_k$, the first number $t$ in the sequence $\{1,\varsigma,\varsigma^2,\cdots\}$ satisfying
		\begin{equation}\label{linef}
			F_{\rho_k}(R_{x^k}(t\eta^k+t^2\tilde{\eta}^k))\leq F_{\rho_k}(x^k)+\sigma t\langle{\rm grad}F_{\rho_k}(x^k),\eta^k\rangle,
		\end{equation}
		\begin{equation}\label{lineg}
			c_i(R_{x^k}(t\eta^k+t^2\tilde{\eta}^k))< 0, \quad\forall i\in \mathcal{L}.
		\end{equation}
		\State Set $x^{k+1}=R_{x^k}(t_k\eta^k+t_k^2\tilde{\eta}^k)$ and $\rho_{k+1}=\rho_k$.
		\State Generate a new symmetric definite positive linear operator $\mathcal{H}_{k+1}$.
		\State Update $\mu^{k+1}_i$ for all $i\in \mathcal{L}$ by
		\begin{equation}\label{gamma}
			\mu_i^{k+1}=\min\{\max\{\lambda_i^{k0},\|\eta^k\|\},\bar{\mu}\}.
		\end{equation}
		\EndFor
	\end{algorithmic}
\end{algorithm}
\begin{remark}
	\begin{description}
		\item[(i)] The update strategy for the penalty parameter $\rho_k$ in steps 3 and 4 of Algorithm \ref{algo1} is adapted from \cite{tits2003primal}, which is critical for ensuring global convergence. We will show that $\rho_k$  can be fixed after a finite number of iterations (see Lemma \ref{facanshu} below).
		\item[(ii)] In \cite{qi2000new}, the algorithm terminates when $\eta^{k0}=0$ (in our notation), and the current strictly feasible iterate $x^k$ is an unconstrained stationary point of $f$, i.e., $\nabla f(x^k)=0$  (see \cite[Lem. 2.2]{qi2000new}). In our Algorithm \ref{algo1}, if $\mathcal{E}=\emptyset$, then the steps 3-6 can be omitted, and the algorithm can also be terminated when $\eta^{k0}=0$, yielding an unconstrained stationary point of $f$ with respect to the inequality constraints, i.e., 
		${\rm grad}f(x^k)=0_x$ and $c_i(x^k)<0$ for all $i\in\mathcal{I}$. However, when $\mathcal{E}\neq\emptyset$, an unconstrained stationary point of $F_{\rho_k}$ cannot be a KKT point of the original problem \eqref{Rpro}, as the equality constraints of \eqref{Rpro} cannot be satisfied. Therefore, in the case where
		$\mathcal{E}\neq\emptyset$ and $\eta^{k0}=0$, Algorithm \ref{algo1} cannot be terminated, and the penalty parameter $\rho_k$ must be increased. 
		In the remainder of this paper, for generality, we always assume that $\mathcal{E}\neq\emptyset$.
		\item[(iii)]  A straightforward calculation shows that $(\eta^k,\lambda^k)$ defined by \eqref{direction} satisfies the following system:
		\begin{subnumcases}
			{\label{s4}}
			\label{s41}
			\mathcal{H}_k[\eta]+\sum_{i\in L}\lambda_i{\rm grad}c_i(x^k)=-{\rm grad}F_{\rho_k}(x^k),\\
			\label{s42}
			\alpha_i^k\langle{\rm grad}c_i(x^k),\eta\rangle-\sqrt{2}\beta_i^k\lambda_i=\alpha_i^k(\min\{\lambda_i^{k0},0\})^3-\theta_k\alpha_i^k\|\eta^{k1}\|^\nu,\quad \forall i\in \mathcal{L}.
		\end{subnumcases}
		\item[(iv)]  For the global convergence of Algorithm \ref{algo1}, steps 10-13 can be omitted by simply setting  $\tilde{\eta}^k=0_{x^k}$. 
		These steps are included solely to avoid the Maratos effect.
		
	\end{description}
\end{remark}

\subsection{Well-definiteness of the algorithm} ~\label{sec:03}

In this subsection, we show that (i) the linear systems \eqref{s1}, \eqref{s2}, and \eqref{s3} admit unique solutions; (ii) the arc search defined by \eqref{linef} and \eqref{lineg} terminates in finite iterations.
For readability, the proofs of the following lemmas are deferred to Appendix \ref{AppendixA}.

We define the linear operator $\mathcal{A}_k$ from $T_{x^k}\mathcal{M}\times\mathbb{R}^{m+\ell}$ to itself as follows: 
\begin{equation}\label{operator}
	\mathcal{A}_k(\eta,\lambda)
	=\left(\mathcal{H}_k[\eta]+\sum_{i\in \mathcal{L}}\lambda_i{\rm grad}c_i(x^k),
	\Lambda_k\right)
\end{equation}
where $\Lambda_k=(\alpha_i^k\langle{\rm grad}c_i(x^k),\eta\rangle-\sqrt{2}\beta_i^k\lambda_i,\ i\in \mathcal{L} )\in \mathbb{R}^{m+\ell}$.
Note that when $\mathcal{M}=\mathbb{R}^n$ and $\mathcal{E}=\emptyset$, the operator $\mathcal{A}_k$ reduces to the coefficient matrix $V_k$ defined in \eqref{V_k}. 
Moreover, it can be observed from its definition that the operator $\mathcal{A}_k$ 
is linear transformation.

The following lemma establishes the nonsingularity of the linear operator for every $k\in\mathbb{N}$, an important property for our analysis. 
\begin{lemma}\label{le1}
	The operator $\mathcal{A}_k$ is nonsingular for all $k\in\mathbb{N}$.
\end{lemma}

Lemma \ref{le1} implies that the systems \eqref{s1}, \eqref{s2}, and \eqref{s3} each have a unique solution. More precisely, for each $k\in \mathbb{N}$, we have
\begin{subequations}\label{relationship}
	\begin{align}
		\label{relationship1}
		&(\eta^{k0},\lambda^{k0})
		=\mathcal{A}_k^{-1}
		\left(-{\rm grad}F_{\rho_k}(x^k),0\right)
		,\\
		\label{relationship2}
		&(\eta^{k1},\lambda^{k1})
		=\mathcal{A}_k^{-1}
		\left(-{\rm grad}F_{\rho_k}(x^k),v^{k1}\right)
		,\\
		\label{relationship3}
		&(\eta^{k2},\lambda^{k2})
		=\mathcal{A}_k^{-1}
		\left(-{\rm grad}F_{\rho_k}(x^k),v^{k2}\right)
		,
	\end{align}
\end{subequations}
where 
\begin{equation}\label{v1}
	v^{k1}=\left(\alpha_i^k(\min\{\lambda_i^{k0},0\})^3, \ i\in \mathcal{L}\right)\in \mathbb{R}^{m+\ell},
\end{equation}
\begin{equation}\label{v2}
	v^{k2}=\left(\alpha_i^k(\min\{\lambda_i^{k0},0\})^3-\alpha_i^k\|\eta^{k1}\|^\nu, \ i\in \mathcal{L}\right)\in \mathbb{R}^{m+\ell}.
\end{equation}

%

As shown before, $\eta^{k0}$ is a descent direction of $F_{\rho^k}$ at $x^k$. Based on this fact, we prove in Lemma~\ref{le3} that $\eta^{k1}$ and $\eta^{k}$ are also descent directions of $F_{\rho^k}$ at $x^k$. 
\begin{lemma}\label{le3}
	Let $\eta^{k1}$ and $\eta^k$ be generated by Algorithm \ref{algo1} at iteration $k$; then they are descent directions for $F_{\rho_k}(x)$ at $x^k$, i.e.,
	\begin{description}
		\item[(i)] $\langle{\rm grad}F_{\rho_k}(x^k),\eta^{k1}\rangle=\langle{\rm grad}F_{\rho_k}(x^k),\eta^{k0}\rangle-\mathop\sum\limits_{i:\lambda_i^{k0}<0}(\lambda_i^{k0})^4<0$;
		\item [(ii)] $\langle{\rm grad}F_{\rho^k}(x^k),\eta^{k}\rangle\leq\tau\langle{\rm grad}F_{\rho^k}(x^k),\eta^{k1}\rangle<0.$
	\end{description}
\end{lemma}

In the final part of this section, we show that the arc search in Algorithm~\ref{algo1} is well-define.


\begin{lemma}\label{le4}
	The arc search procedure defined by \eqref{linef} and \eqref{lineg} can be terminated after a finite number of iterations.
\end{lemma}

\section{Global convergence}\label{Sec:global}

In this section, we establish the global convergence of Algorithm \ref{algo1} under some reasonable assumptions.
In what follows, we assume that Algorithm \ref{algo1} generates an infinite sequence $\{x^k\}\subset{\rm int}\mathcal{F}$. 
For the sake of simplicity, we define
$\mathcal{I}(x)=\{i\in \mathcal{I}\mid c_i(x)=0\}$, $\mathcal{E}(x)=\{i\in \mathcal{E}\mid c_i(x)=0\}$, and $\mathcal{L}(x)=\{i\in \mathcal{L}\mid c_i(x)=0\}$.
First, we make the following conventional assumptions, which are commonly used in the related literature (see, e.g., \cite{qi2000new,tits2003primal,obara2022sequential}).

\begin{assumption}\label{A1}
	For all $x\in\mathcal{F}$, the LICQ of problem \eqref{TRpro} holds at $x$, i.e., the set of vectors $\{{\rm grad}c_i(x)\mid i\in \mathcal{L}(x)\}$ is linearly independent on $T_x\mathcal{M}$.
\end{assumption}
\begin{assumption}\label{A2}
	The sequence $\{x^k\}$ generated by Algorithm \ref{algo1} is bounded.
\end{assumption}
\begin{assumption}\label{A3}
	There exist two positive constants $a_1$ and $a_2$ such that, for any $k$, $a_1\|\xi\|^2\leq\langle\mathcal{H}_k[\xi],\xi\rangle\leq a_2\|\xi\|^2$ holds for all $\xi \in T_{x^k}\mathcal{M}$.
\end{assumption}
\begin{assumption}\label{A3.5}
	For any $x\in\mathcal{F}\backslash\mathcal{F}_P$, there exist no scalars $u_i\geq 0, i\in \mathcal{I}(x)$, and $v_i\geq 0, i\in \mathcal{E}(x)$, such that
	$$\sum_{i\in \mathcal{E}}{\rm grad}c_i(x)-\sum_{i\in \mathcal{I}(x)}u_i{\rm grad}c_i(x)-\sum_{i\in \mathcal{E}(x)}v_i{\rm grad}c_i(x)=0_{x}.$$
\end{assumption}

\begin{remark}
	When $x\in\mathcal{F}\backslash\mathcal{F}_P$, it is clear that Assumption \ref{A1} is weaker than the conventional assumption, which requires the set $\{{\rm grad}c_i(x)\mid i\in\mathcal{I}(x)\cup\mathcal{E}\}$ to be linearly independent. 
	Assumption \ref{A3.5} is only used to prove Lemma \ref{facanshu}. Moreover, if the set $\{{\rm grad}c_i(x)\mid i\in\mathcal{I}(x)\cup\mathcal{E}\}$ is linearly independent at $x\in\mathcal{F}\backslash\mathcal{F}_P$, then Assumption \ref{A3.5} holds naturally.
\end{remark}

The following lemma implies that the penalty parameter $\rho_k$ can be fixed after a finite number of iterations, and its proof is a variant of the one in \cite[Lem. 4.1]{tits2003primal}. For completeness, we give the proof in Appendix \ref{AppendixB}.
\begin{lemma}\label{facanshu}
	Suppose that Assumptions \ref{A1}-\ref{A3.5} hold. Then the penalty parameter $\rho_k$ in Algorithm \ref{algo1} is updated only a finite number of times.
\end{lemma}

From Lemma \ref{facanshu}, we know there exist an index $\bar{k}\geq 0$ and
a constant $\bar{\rho}>0$ such that $\rho_k\equiv \bar{\rho}$ for all $k\geq \bar{k}$.
Without loss of generality, we may assume that $\bar{k}=0$ in the subsequent discussion. Thus, the problem \eqref{TRpro} refers specifically to the case of $\rho\equiv \bar{\rho}$. Accordingly, we can also assume that $\eta^{k0}\neq 0_{x^k}$ for all $k$; otherwise $\rho_k$ would be increased.
In addition, from Lemma \ref{le3} (ii) and \eqref{gamma}, we know that 
$\mu_i^k>0, i\in \mathcal{L}$, for all $k$, which further implies that $\alpha_i^k>0, i\in \mathcal{L}$. 

We have already proved (in Lemma \ref{le1}) that $\mathcal{A}_k$ is nonsingular for all $k$. Therefore $\mathcal{A}_k^{-1}$ exists for all $k$. 
We next establish the boundedness of the sequence $\{\|\mathcal{A}_k^{-1}\|_{\rm op}\}$ in Lemma~\ref{le5}. Since the proof follows the same lines as Lemma 3.1 in \cite{qi2000new}, we omit the details here. Based on Lemma~\ref{le5}, Lemma~\ref{le6} further establishes the boundedness of the solutions to the three linear systems \eqref{s1}, \eqref{s2}, and \eqref{s3}. To analyze the global convergence, we first present Lemma~\ref{le8} and Lemma~\ref{le9}, which guarantee that any accumulation point of $\{x^k\}$ is a KKT point of problem \eqref{TRpro}. The proofs of Lemmas~\ref{le6}--\ref{le9} are deferred to Appendix~\ref{AppendixB}.

\begin{lemma}\label{le5}
	Suppose that Assumptions \ref{A1}-\ref{A3.5} hold. Then the sequence $\{\|\mathcal{A}_k^{-1}\|_{\rm op}\}$ is bounded.
\end{lemma}
\begin{lemma}\label{le6} Suppose that Assumptions \ref{A1}-\ref{A3.5} hold. Then we have (i) the sequences $\{(\eta^{k0},\lambda^{k0})\}$, $\{(\eta^{k1},\lambda^{k1})\}$ and $\{(\eta^{k2},\lambda^{k2})\}$ are all bounded; (ii) there exists a constant $C>0$ such that $\|\eta^k-\eta^{k1}\|\leq C\|\eta^{k1}\|^\nu,\ \forall k\in\mathbb{N}$.
\end{lemma}
\begin{lemma}\label{le7}
	Suppose that Assumptions \ref{A1}-\ref{A3.5} hold. If there is an index set $\mathcal{K}\subseteq \mathbb{N}$ such that $\{x^k\}_{\mathcal{K}}\rightarrow x^*$ and $\{\eta^k\}_{\mathcal{K}}\rightarrow 0_{x^*}$, then $x^*$ is a KKT point of problem \eqref{TRpro}, and the sequence $\{\lambda^{k0}\}_{\mathcal{K}}$ converges to the unique KKT multiplier corresponding to $x^*$.
\end{lemma}
\begin{lemma}\label{le8}
	Suppose that Assumptions \ref{A1}-\ref{A3.5} hold. Let $\mathcal{K}\subseteq\mathbb{N}$ such that $\{x^k\}_{\mathcal{K}}\rightarrow x^*$. If $\{\|\eta^{k-1}\|\}_{\mathcal{K}}\rightarrow 0$, then $x^*$ is a KKT point of problem \eqref{TRpro}.
\end{lemma}
\begin{lemma}\label{le9}
	Suppose that Assumptions \ref{A1}-\ref{A3.5} hold. Let $\mathcal{K}\subseteq\mathbb{N}$ such that $\{x^k\}_{\mathcal{K}}\rightarrow x^*$. If $\inf\{\|\eta^{k-1}\|\}_{\mathcal{K}}> 0$, then $x^*$ is a KKT point of problem \eqref{TRpro}.
\end{lemma}

Based on the above lemmas, we can establish the global convergence of Algorithm \ref{algo1}.

\begin{theorem}\label{th2}
	Suppose that Assumptions \ref{A1}-\ref{A3.5} hold. If there exists an infinite index set $\mathcal{K}\subseteq\mathbb{N}$ such that $\{(x^k,\lambda^{k0})\}_{\mathcal{K}}\rightarrow (x^*,\lambda^*)$, then either 
	$(x^*,\lambda^*)$ is a KKT pair of problem \eqref{TRpro}, or there exists an infinite index $\mathcal{K}'\subseteq\mathcal{K}$ such that $\{\lambda^{(k-1)0}\}_{\mathcal{K}'} \rightarrow \lambda^{**}$ and $(x^*,\lambda^{**})$ is a KKT pair of problem \eqref{TRpro}. Furthermore, $x^*$ is a KKT point of problem \eqref{Rpro}.
\end{theorem}
\begin{proof}
	It follows from Lemmas \ref{le8} and \ref{le9} that $x^*$ is a KKT point of problem \eqref{TRpro}. The rest of proof can be divided into two parts.
	
	In the first part, if there exists some index $i\in \mathcal{L}(x^*)$ and an infinite index set $\mathcal{K}'\subseteq\mathcal{K}$ such that $\{\mu^k_i\}_{\mathcal{K}'}\rightarrow 0$. Then by \eqref{gamma}, we have $\{\|\eta^{k-1}\|\}_{\mathcal{K}'}\rightarrow 0$. Similar to the proof of Lemma \ref{le7}, we know that $\{\|\eta^{(k-1)0}\|\}_{\mathcal{K}'}\rightarrow 0$. Without loss of generality, we assume that $\{x^{k-1}\}_{\mathcal{K}'}\rightarrow x^{**}$. Then the proof of  Lemma \ref{le7} and Lemma \ref{le8} show that $\lambda^{**}=\lim_{k\in \mathcal{K}'}\lambda^{(k-1)0}\geq0$ and $x^*=x^{**}$, with  $(x^*,\lambda^{**})$ being a KKT pair of problem \eqref{TRpro}. This implies that the conditions (a) and (c) in step 3 of Algorithm \ref{algo1} hold at the iteration $k-1$ for $k\in \mathcal{K}'$ large enough. Since $\rho_k\equiv\bar{\rho}<+\infty$ for all $k$, thus we must have $\lambda^{**}_\mathcal{E}\geq r_2e>0$. On the other hand, since $x^{**}$ is a KKT point of problem \eqref{TRpro}, we know that $\lambda^{**}_ic_i(x^{**})=0$ for all $i\in\mathcal{E}$. Therefore, $c_i(x^*)=c_i(x^{**})=0$ for all $i\in\mathcal{E}$. Then Proposition \ref{tp-p} shows that $x^*$ is a KKT point of problem \eqref{Rpro}.
	
	In the second part, if for any $i\in \mathcal{L}(x^*)$, there is no infinite index $\mathcal{K}'\subseteq\mathcal{K}$ such that $\{\mu^k_i\}_{\mathcal{K}'}\rightarrow 0$. Combining with the definition of $\mu^k$, without loss of generality, we assume that $\{\mu^k_i\}_{\mathcal{K}}\rightarrow \mu^*_i>0$ for all $i\in \mathcal{L}(x^*)$. On the other hand, Lemma \ref{le6} implies that $\{\eta^{k0}\}_{\mathcal{K}}$ is bounded. Then by Assumption \ref{A3} and the definitions of $\alpha_i^k$ and $\beta_i^k$, we can also assume that  $\{\eta^{k0}\}_{\mathcal{K}}\rightarrow \eta^*$, $\{\mathcal{H}_k\}_{\mathcal{K}}\rightarrow \mathcal{H}_*$, $\{\alpha_i^k\}_{\mathcal{K}}\rightarrow\alpha_i^*$, and $\{\beta_i^k\}_{\mathcal{K}}\rightarrow\beta_i^*$.
	Then for all $i\in \mathcal{L}(x^*)$, we have
	$$\{\alpha_i^k\}_{\mathcal{K}}\rightarrow\alpha_i^*=\dfrac{c_i(x^*)}{\sqrt{c_i^2(x^*)+(\mu_i^*)^2}}+1>0,$$ $$\{\beta_i^k\}_{\mathcal{K}}\rightarrow\beta_i^*=\left(1-\dfrac{\mu_i^*}{\sqrt{c_i^2(x^*)+(\mu_i^*)^2}}\right)^{1/2}=0.$$
	Let $\bar{\lambda}^*$ be a KKT multiplier corresponding to the KKT point $x^*$ of problem \eqref{TRpro}. We will show that ${\lambda}^*=\bar{\lambda}^*$. From the above analysis, it is clear that $(0_{x^*},\bar{\lambda}^*)$ is the solution to the following linear system. 
	\begin{equation}\label{th4.1pro1}
		\mathcal{A}_*(
		\eta,\lambda)
		:=
		\left(\mathcal{H}_*[\eta]+\sum_{i\in \mathcal{L}}\lambda_i{\rm grad}c_i(x^*),\Lambda_*\right)
		=\left(-{\rm grad}F_{\bar{\rho}}(x^*),0\right),
	\end{equation}
	where $\Lambda_*=\left(\alpha_i^*\langle{\rm grad}c_i(x^*),\eta\rangle-\sqrt{2}\beta_i^*\lambda_i,\ i\in \mathcal{L}\right)\in \mathbb{R}^{m+\ell}$.
	Let $(\eta',\lambda')$ be a solution to the linear system $\mathcal{A}_*(\eta,\lambda)=(0_{x^*},0)$, i.e.,
	\begin{eqnarray}\label{proof-th2-1}
		\label{proof-th2-11}
		&&\mathcal{H}_*[\eta']+\sum_{i\in \mathcal{L}}\lambda'_i{\rm grad}c_i(x^*)=0_{x^*},\\
		\label{proof-th2-12}
		&&\alpha_i^*\langle{\rm grad}c_i(x^*),\eta'\rangle-\sqrt{2}\beta_i^*\lambda'_i=0,\quad i\in \mathcal{L}.
	\end{eqnarray}
	It follows from \eqref{alha-beta} that 
	\begin{equation}\label{proof-th2-2}
		\alpha_i^*\in[0,1],\,\,\beta_i^*\in[0,\sqrt{2}],\,\,\,\, (\alpha_i^*)^2+(\beta_i^*)^4\geq 3-2\sqrt{2}.
	\end{equation}
	If $\beta^*_i=0$, then by \eqref{proof-th2-2}, we have $\alpha_i^*\neq 0$. Combining with \eqref{proof-th2-12}, we know that $\langle{\rm grad}c_i(x^*),\eta'\rangle=0$. Therefore, from \eqref{proof-th2-11} and \eqref{proof-th2-12}  we have 
	$$\left\langle \eta',\mathcal{H}_*[\eta']\right\rangle+\sum_{i\in \mathcal{L}:\beta^*_i>0}\dfrac{\alpha_i^*\arrowvert\langle{\rm grad}c_i(x^*),\eta'\rangle\arrowvert^2}{\sqrt{2}\beta_i^*}=0.$$
	Thus $\left\langle \eta',\mathcal{H}_*[\eta']\right\rangle=0$. By Assumption \ref{A3}, we have $\eta'=0_{x^*}$. Furthermore, from Assumption \ref{A1} and \eqref{proof-th2-11}, \eqref{proof-th2-12}, we obtain $\lambda'=0$. Thus, it follows that $\mathcal{A}_*$ is nonsingular, and the solution to system \eqref{th4.1pro1} is unique. Since $\{(x^k,\lambda^{k0})\}_{\mathcal{K}}\rightarrow (x^*,\lambda^*)$, by taking the limit on $\mathcal{K}$ for both sides of the system \eqref{s1}, we know that $(\eta^*,\lambda^*)$ is also the solution to \eqref{th4.1pro1}. Thus, we have $\lambda^*=\bar{\lambda}^*\geq 0$ and $\eta^*=0_{x^*}$. This implies that $(x^*,\lambda^*)$ is a KKT pair of problem \eqref{TRpro}. Similar to the first part of this proof, we can also show that $x^*$ is a KKT point of problem \eqref{Rpro}.
\end{proof}

\section{Strong and superlinear convergence}\label{Sec:superlinear}
In this section, we establish the strong and superlinear convergence of Algorithm \ref{algo1}.
Let $\{(x^k,\lambda^{k0})\}_{\mathcal{K}}\rightarrow (x^*,\lambda^*)$ for some infinite index set $\mathcal{K}\subseteq\mathbb{N}$.
Denote by $\hat{\lambda}^*$ the KKT multiplier corresponding to the KKT point $x^*$ of problem \eqref{TRpro}. From Theorem \ref{th2}, it follows that $\hat{\lambda}^*={\lambda}^*$ or $\hat{\lambda}^*={\lambda}^{**}$,
and that $x^*$ is a KKT point of problem \eqref{Rpro}. 
Let $\bar{\lambda}^*$ be the KKT multiplier corresponding to the KKT point $x^*$ of problem \eqref{Rpro}. 
Under some additional assumptions, we first show that $(x^*,\lambda^*)$ is a KKT pair of problem \eqref{TRpro}. We then show that the whole sequence $\{(x^k,\lambda^{k0})\}$ converges to $(x^*,\lambda^*)$, i.e., Algorithm \ref{algo1} is strongly convergent. Finally, we will prove that the sequence $\{x^k\}$ is superlinearly convergent.

\begin{assumption}\label{A4}
	The strict complementarity condition of problem \eqref{Rpro} holds at $(x^*,\bar{\lambda}^*)$, i.e., $\bar{\lambda}_i^*c_i(x^*)=0$ and  $\bar{\lambda}_i^*-c_i(x^*)>0$ for all $i\in \mathcal{I}$.
\end{assumption}
\begin{assumption}\label{A5}
	The SOSC of problem \eqref{Rpro} holds at $(x^*,\bar{\lambda}^*)$.
\end{assumption}
\begin{assumption}\label{A6}
	The constant $\bar{\mu}$ is chosen to be sufficiently large such that $\lambda_i^*<\bar{\mu}$ for all $i\in \mathcal{L}$.
\end{assumption}

\begin{remark}
	Note that Assumptions \ref{A4} and \ref{A5} are commonly used to establish local fast convergence (see, e.g., \cite{obara2022sequential}). Assumption \ref{A6} is mild, as Assumption \ref{A1} implies that $\lambda^*$ is unique and therefore bounded. 
\end{remark}

In Lemma~\ref{Atp} below, we verify that the strict complementarity condition and the SOSC hold for Problem~\eqref{TRpro} at $(x^*,\hat{\lambda}^*)$. 
Moreover, Lemma~\ref{le9.4} proves that $\hat{\lambda}^* = \lambda^*$. The proofs of these two lemmas can be found in Appendix~\ref{AppendixC}.
\begin{lemma}\label{Atp}
	Suppose that Assumptions \ref{A1}-\ref{A5} hold. Then we have 
	\begin{description}
		\item[(i)] the strict complementarity condition of problem \eqref{TRpro} holds at $(x^*,\hat{\lambda}^*)$, i.e., $\hat{\lambda}_i^*c_i(x^*)=0$ and $\hat{\lambda}_i^*-c_i(x^*)>0$ for all $i\in \mathcal{L}$;
		\item [(ii)] the SOSC of problem \eqref{TRpro} holds at $(x^*,\hat{\lambda}^*)$.
	\end{description}
\end{lemma}	

\begin{lemma}\label{le9.4}
	Suppose that Assumptions \ref{A1}-\ref{A4} hold. Then $(x^*,\lambda^*)$ is a KKT pair of problem \eqref{TRpro}, i.e., $\hat{\lambda}^*={\lambda}^*$.
\end{lemma}

In the rest of this paper, for simplicity, we use the notation ${\lambda}^*$ instead of $\hat{\lambda}^*$. To analyze the strong and superlinear convergence of the sequence $\{x^k\}$, we need to locally transform problem \eqref{TRpro} into a Euclidean space. For this, we introduce the concept of Riemannian normal coordinate system. 
Let $E:\mathbb{R}^n\rightarrow T_{x^*}\mathcal{M}$ be a linear bijection such that $\{E(e_j)\}_{j=1}^n$ forms an orthogonal basis of $T_{x^*}\mathcal{M}$, where $e_j$ denotes  the $j$-th unit vector in $\mathbb{R}^n$. 
We know that there exist a neighborhood $\mathcal{U}$ of $x^*$ and a neighborhood $V$ of $0_{x^*}$ such that the exponential mapping ${\rm Exp}_{x^*}:V\rightarrow \mathcal{U}$ is a diffeomorphism (see \cite[Sec. 5.4]{absil2008optimization}). Let $\varphi=E^{-1}\circ{\rm Exp}_{x^*}^{-1}$. It is clear that $(\mathcal{U},\varphi)$ is a chart,  which is known as a Riemannian normal coordinate system. Letting $x\in\mathcal{U}$ and using this chart, one can define: $\hat{x}=\varphi(x)$, $\hat{\xi}_{\hat{x}}={\rm D}\varphi(x)[\xi_x]$, $\hat{f}=f\circ\varphi^{-1}$, $\hat{c}_i=c_i\circ\varphi^{-1}$ for all $i\in \mathcal{L}$. It follows from \cite[Sec. 3.6]{absil2008optimization} that 
\begin{equation}\label{inner}
	\langle\xi_x,\zeta_x\rangle=\hat{\xi}_{\hat{x}}^{\top}G_{\hat{x}}\hat{\zeta}_{\hat{x}}
\end{equation}
for all $\xi_x,\zeta_x\in T_x\mathcal{M}$, where $G_{\hat{x}}$ 
is the matrix representation of the Riemannian metric at $x$, which is symmetric positive definite. Note that $G_{\hat{x}^*}=I_n$, where $I_n$ denotes the identity matrix of order $n$ (see, e.g., \cite{klingenberg1982riemannian,yang2014optimality}).

Now, we locally transform problem \eqref{TRpro} into the Euclidean space $\mathbb{R}^n$
as follows:
\begin{equation}\label{Epro}
	\begin{aligned}
		&\min\hat{F}_{\bar{\rho}}(\hat{x})\\
		&{\rm \,\,s.t.}\,\,\, \hat{c}_i(\hat{x})\leq 0,\quad i\in \mathcal{L},\\
		&\quad\quad\,\hat{x}\in\varphi(\mathcal{U})\subseteq\mathbb{R}^n.
	\end{aligned}\tag{{\rm EP}}
\end{equation}
Clearly, for a point $\hat{x}\in\varphi({\mathcal{U}})$, the KKT conditions of problem \eqref{Epro} at $\hat{x}$ are as follows:
\begin{eqnarray*}
	&&\nabla \hat{F}_{\bar{\rho}}(\hat{x})+\sum_{i\in \mathcal{L}}\lambda_i\nabla \hat{c}_i(\hat{x})=0,\\
	&&\lambda_i\geq 0,\,\,\hat{c}_i(\hat{x})\leq 0\quad \forall\, i\in \mathcal{L},\\
	&&\lambda_i\hat{c}_i(\hat{x})=0,\quad \forall\, i\in \mathcal{L}.
\end{eqnarray*}

Note that $\widehat{{\rm grad}{F}_{\bar{\rho}}}(\hat{x})=G_{\hat{x}}^{-1}\nabla\hat{F}_{\bar{\rho}}(\hat{x})$ (see \cite[Sec. 3.6]{absil2008optimization}). 
Then the following lemma is obvious, so we omit its proof.
\begin{lemma}\label{le9.5}
	Let $x'\in\mathcal{U}$. Then $(x',\lambda')$ is a KKT pair of problem \eqref{TRpro} if and only if $(\hat{x}',\lambda')$ is a KKT pair of problem \eqref{Epro}. In particular, $(\hat{x}^*,\lambda^*)$ is a KKT pair of \eqref{Epro}.
\end{lemma}

Denote $\hat{L}_{\bar{\rho}}(\hat{x},\lambda)=\hat{F}_{\bar{\rho}}(\hat{x})+\sum_{i\in \mathcal{L}}\lambda_i\hat{c}_i(\hat{x})$. The following lemma characterizes the relationship between the Riemannian Hessian ${\rm Hess}_xL_{\bar{\rho}}(x^*,\lambda^*)$ and the Euclidean Hessian matrix $\nabla^2_{\hat{x}}\hat{L}_{\bar{\rho}}(\hat{x}^*,\lambda^*)$.
\begin{lemma}\label{le10}
	Under the Riemannian normal coordinate system $(\mathcal{U},\varphi)$, the following relation holds:
	$${\rm D}\varphi(x^*){\rm Hess}_xL_{\bar{\rho}}(x^*,\lambda^*){\rm D}\varphi^{-1}(\hat{x}^*)=\nabla^2_{\hat{x}}\hat{L}_{\bar{\rho}}(\hat{x}^*,\lambda^*).$$
\end{lemma}
\begin{proof}
	Denote $L_{\bar{\rho},\lambda}(x)=L_{\bar{\rho}}(x,\lambda)$ and $\hat{L}_{\bar{\rho},\lambda}(\hat{x})=\hat{L}_{\bar{\rho}}(\hat{x},\lambda)$. Then we have ${\rm Hess}_xL_{\bar{\rho}}(x^*,\lambda^*)={\rm Hess}L_{\bar{\rho},\lambda^*}(x^*)$ and $\nabla^2_{\hat{x}}\hat{L}_{\bar{\rho}}(\hat{x}^*,\lambda^*)=\nabla^2\hat{L}_{\bar{\rho},\lambda^*}(\hat{x}^*)$.
	
	It follows from \cite[Prop. 5.5.4]{absil2008optimization} that $${\rm Hess}L_{\bar{\rho},\lambda^*}(x^*)={\rm Hess}(L_{\bar{\rho},\lambda^*}\circ{\rm Exp}_{x^*})(0_{x^*}),$$ where $$L_{\bar{\rho},\lambda^*}\circ{\rm Exp}_{x^*}=F_{\bar{\rho}}\circ{\rm Exp}_{x^*}+\sum_{i\in \mathcal{L}}\lambda^*_ic_i\circ{\rm Exp}_{x^*}:T_{x^*}\mathcal{M}\rightarrow\mathbb{R},$$ 
	and ${\rm Hess}(L_{\bar{\rho},\lambda^*}\circ{\rm Exp}_{x^*})(0_{x^*})$ denotes the Euclidean Hessian of $L_{\bar{\rho},\lambda^*}\circ{\rm Exp}_{x^*}$ at the origin of $T_{x^*}\mathcal{M}$. Therefore, we only need to show that 
	$${\rm D}\varphi(x^*){\rm Hess}(L_{\bar{\rho},\lambda^*}\circ{\rm Exp}_{x^*})(0_{x^*}){\rm D}\varphi^{-1}(\hat{x}^*)=\nabla^2\hat{L}_{\bar{\rho},\lambda^*}(\hat{x}^*).$$
	Denote $E_j=E(e_j)$, then for any $y=\sum_{j=1}^{n}y^jE_j\in T_{x^*}\mathcal{M}$, we have $$L_{\bar{\rho},\lambda^*}\circ{\rm Exp}_{x^*}(y)=L_{\bar{\rho},\lambda^*}\circ{\rm Exp}_{x^*}(y^1E_1+\cdots,+y^nE_n)=L_{\bar{\rho},\lambda^*}\circ\varphi^{-1}(y^1,\cdots,y^n).$$
	Thus 
	$${\rm Hess}(L_{\bar{\rho},\lambda^*}\circ{\rm Exp}_{x^*})(0_{x^*})[z]=\sum_{j,l}\partial_{j,l}^2(L_{\bar{\rho},\lambda^*}\circ\varphi^{-1})(0,\cdots,0)z^lE_j,$$
	where $z=\sum_{j=1}^{n}z^jE_j$ and $\partial_{j,l}^2(L_{\bar{\rho},\lambda^*}\circ\varphi^{-1})=\dfrac{\partial}{\partial_j\partial_l}(L_{\bar{\rho},\lambda^*}\circ\varphi^{-1})$ (see \cite[Sec. 5.5]{absil2008optimization}). 
	
	On the other hand, let $y=\sum_{j=1}^{n}y^jE_j\in T_{x^*}\mathcal{M}$ be arbitrary. Then
	\begin{equation*}
		\begin{aligned}
			{\rm D}\varphi(x^*)[y]&={\rm D}(E^{-1}\circ{\rm Exp}_{x^*}^{-1})(x^*)\left[y^1E_1+\cdots+y^nE_n\right]\\
			&={\rm D}E^{-1}({\rm Exp}_{x^*}^{-1}(x^*))\left[{\rm D}{\rm Exp}_{x^*}^{-1}(x^*)\left[y^1E_1+\cdots+y^nE_n\right]\right]\\
			&={\rm D}E^{-1}(0_{x^*})\left[y^1E_1+\cdots+y^nE_n\right]\\
			&=\sum_{j=1}^{n}y^je_j.
		\end{aligned}
	\end{equation*}
	Therefore, for any vector $v=\sum_{j=1}^nv^je_j\in\mathbb{R}^n$, we have
	\begin{equation*}
		\begin{aligned}
			&\quad\,\,{\rm D}\varphi(x^*){\rm Hess}(L_{\bar{\rho},\lambda^*}\circ{\rm Exp}_{x^*})(0_{x^*}){\rm D}\varphi^{-1}(\hat{x}^*)[v]\\
			&={\rm D}\varphi(x^*){\rm Hess}(L_{\bar{\rho},\lambda^*}\circ{\rm Exp}_{x^*})(0_{x^*})\left[v^1E_1+\cdots+v^nE_n\right]\\
			&={\rm D}\varphi(x^*)\sum_{j,l}\partial_{j,l}^2(L_{\bar{\rho},\lambda^*}\circ\varphi^{-1})(0,\cdots,0)v^lE_j\\
			&=\sum_{j,l}\partial_{j,l}^2(L_{\bar{\rho},\lambda^*}\circ\varphi^{-1})(0,\cdots,0)v^le_j\\
			&=\nabla^2\hat{L}_{\bar{\rho},\lambda^*}(\hat{x}^*)v,
		\end{aligned}
	\end{equation*}
	which completes the proof.
\end{proof}

Through the above analysis, the following lemma establishes that the relevant conditions hold for problem \eqref{Epro}.
\begin{lemma}\label{pro3}
	Suppose that Assumptions \ref{A1}-\ref{A5} hold. Let $(\mathcal{U},\varphi)$ be a Riemannian normal coordinate system. Denote $\hat{\mathcal{L}}(\hat{x}):=\{i\in \mathcal{L}\mid \hat{c}_i(\hat{x})=0\}$, then we have 
	\begin{description}
		\item[(i)] for all $\hat{x}\in\varphi(\mathcal{U}\cap\mathcal{F})$, the LICQ of problem \eqref{Epro}) holds at $\hat{x}$, i.e., the set $\{\nabla\hat{c}_i(\hat{x})\mid i\in \hat{\mathcal{L}}(\hat{x})\}$
		is linearly independent;
		\item [(ii)] the strict complementarity condition of problem \eqref{Epro} holds at $(\hat{x}^*,\lambda^*)$, i.e.,  $\lambda_i^*\hat{c}_i(\hat{x}^*)=0$ and $\lambda_i^*-\hat{c}_i(\hat{x}^*)>0$ for all $i\in \mathcal{L}$;
		\item [(iii)] the SOSC of problem \eqref{Epro} holds at $(\hat{x}^*,\lambda^*)$.
	\end{description}
\end{lemma}
\begin{proof}
	Part (i) can be found in \cite[Sec. 4]{yang2014optimality}. In view of the fact that $\hat{c}_i(\hat{x}^*)=c_i(x^*)$, we obtain part (ii). Next, we prove part (iii) as follows. For any nonzero vector $u\in \{v\mid v^{\top}\nabla\hat{c}_i(\hat{x}^*)=0,\,i\in \hat{\mathcal{L}}(\hat{x}^*)\}$, we have $\langle\nabla\hat{c}_i(\hat{x}^*),u\rangle=0$ for all $i\in \hat{\mathcal{L}}(\hat{x}^*)$. Note that $u^{\top}\nabla\hat{c}_i(\hat{x}^*)=\left\langle{\rm grad}c_i(x^*),{\rm D}\varphi^{-1}(\hat{x}^*)[u]\right\rangle$ (see \cite[Sec. 3.6]{absil2008optimization}) for all $i\in \hat{\mathcal{L}}(\hat{x}^*)=\mathcal{L}({x}^*)$. Thus 
	$${\rm D}\varphi^{-1}(\hat{x}^*)[u]\in \{\xi_{x^*}\in T_{x^*}\mathcal{M}\mid\langle\xi_{x^*},{\rm grad}c_i(x^*) \rangle=0,\, i\in \mathcal{L}(x^*)\}.$$ 
	Then by Lemma \ref{Atp} (ii), we have 
	\begin{equation*}
		\begin{aligned}
			0&<\left\langle{\rm D}\varphi^{-1}(\hat{x}^*)[u],{\rm Hess}_xL_{\bar{\rho}}(x^*,\lambda^*){\rm D}\varphi^{-1}(\hat{x}^*)[u]\right\rangle\\
			&=u^{\top}{\rm D}\varphi(x^*){\rm Hess}_xL_{\bar{\rho}}(x^*,\lambda^*){\rm D}\varphi^{-1}(\hat{x}^*)[u]\\
			&= u^{\top}\nabla^2_{\hat{x}}\hat{L}_{\bar{\rho}}(\hat{x}^*,\lambda^*)u,
		\end{aligned}
	\end{equation*}
	where the first equality holds due to \eqref{inner} and $G_{\hat{x}^*}=I_n$, and the second equality comes from Lemma \ref{le10}. This implies that part (iii) holds.
\end{proof}

To show that the whole sequence $\{x^k\}$ converges to $x^*$, we need the following proposition. Its original version is due to \cite{more1983computing} and a slightly modified version can be found in \cite[Prop. 5.4]{kanzow1999qp}. We generalize the modified version to metric spaces, and the proof is provided in Appendix \ref{AppendixC}.
\begin{proposition}\label{pro4}
	Let $X$ be a complete metric space, and let $p^*\in X$ be an isolated accumulation point of the sequence $\{p^k\}\subset X$. If for any infinite subset $\mathcal{K}\subseteq \mathbb{N}$ with $\{p^k\}_{\mathcal{K}}\to p^*$, there exists an infinite subset $\mathcal{K}'\subset\mathcal{K}$ such that $\{{\rm dist}(p^{k+1},p^k)\}_{\mathcal{K}'}\rightarrow 0$, where ${\rm dist}(p^{k+1},p^k)$ denotes the distance between $p^{k+1}$ and $p^k$ in $X$. Then the whole sequence $\{p^k\}$ converges to $p^*$.
\end{proposition}

Based on Proposition \ref{pro4}, we can establish the strong convergence of Algorithm \ref{algo1}.
\begin{theorem} \label{le11}
	Suppose that Assumptions \ref{A1}-\ref{A5} hold. Then the whole sequence $\{x^k,\lambda^{k0}\}$ converges to $(x^*,\lambda^*)$.
\end{theorem}
\begin{proof}
	It follows from Lemma \ref{le9.5} that $\hat{x}^*$ is a KKT point of problem \eqref{Epro}. This together with Lemma \ref{pro3} shows that $\hat{x}^*$ is an isolated KKT point of problem \eqref{Epro} (see \cite{robinson1980strongly}). Then from Lemma \ref{le9.5} and the smoothness of $\varphi$, we know that ${x}^*$ is an isolated accumulation of the sequence $\{{x}^k\}$. 
	
	From Proposition \ref{pro4}, to show that the whole sequence $\{x^k\}$ converges to $x^*$, we only need to show that for every sequence $\{x^k\}_{\mathcal{K}}$ converges to $x^*$, there is a subsequence $\{x^k\}_{\mathcal{K}'}$ with ${\mathcal{K}}'\subseteq{\mathcal{K}}$ such that $\{{\rm dist}(x^{k+1},x^k)\}_{\mathcal{K}'}\rightarrow 0$.
	We first show that $\{\|\eta^k\|\}_{\mathcal{K}'}\rightarrow 0$. Suppose by contradiction that 
	\begin{equation}\label{proof-le11-4}
		\mathop{\rm lim\,inf}\limits_{k\in\mathcal{K}'}\|\eta^{k}\|>0.
	\end{equation}
	Then it follows from Lemma \ref{le6} (ii) that $\liminf_{k\in\mathcal{K}'}\|\eta^{k1}\|>0$. By the boundedness of $\{\mathcal{H}_k\}$, $\{\alpha_i^k\}$ and $\{\beta_i^k\}$, $i\in \mathcal{L}$, without loss of generality, we assume that $\{\mathcal{H}_k\}_{\mathcal{K}'}\rightarrow \mathcal{H}^*$, $\{\alpha_i^k\}_{\mathcal{K}'}\rightarrow\alpha_i^*$ and $\{\beta_i^k\}_{\mathcal{K}'}\rightarrow\beta_i^*$, $i\in \mathcal{L}$. Therefore \begin{equation}\label{proof-le11-1}
		\{\mathcal{A}_k\}_{\mathcal{K}'}\rightarrow\mathcal{A}_*.
	\end{equation}
	Similar to the second part of proof of Theorem \ref{th2}, we know that $\mathcal{A}_*$ is nonsingular. On the other hand, since $\{\lambda^{k0}\}_{\mathcal{K}'}\rightarrow\lambda^*\geq 0$, we have  $$\left\{\left(\min\{\lambda_i^{k0},0\}\right)^3\right\}_{\mathcal{K}'}\rightarrow 0, \ \forall i\in \mathcal{L}.$$ 
	It then follows from \eqref{proof-le11-1} and the continuity of ${\rm grad}F_{\bar{\rho}}$ that $\{(\eta^{k1},\lambda^{k1})\}_{\mathcal{K}'}$ converges to the solution of the following linear system:
	\begin{equation}\label{proof-le11-3}
		\mathcal{A}_*(\eta,\lambda)
		=
		\left(-{\rm grad}F_{\bar{\rho}}(x^*),0\right).
	\end{equation}
	Then by the nonsingularity of $\mathcal{A}_*$, we know that the solution of the above system is unique. Moreover, since $(x^*,\lambda^*)$ is a KKT point of problem \eqref{TRpro}, then $(0,\lambda^*)$ is the unique solution of \eqref{proof-le11-3}. Thus we must have $\{\eta^{k1}\}_{\mathcal{K}'}\rightarrow 0_{x^*}$, which is a contradiction. Thus $\{\|\eta^k\|\}_{\mathcal{K}'}\rightarrow 0$. 
	
	Next, since $\|\tilde{\eta}^{k}\|\leq\|\eta^k\|$, we also have $\{\tilde{\eta}^{k}\}_{\mathcal{K}'}\rightarrow 0$. It then follows from the continuity of $R$ and $t_k\in(0,1]$ that
	$$\lim_{k\in\mathcal{K}'}x^{k+1}=\lim_{k\in\mathcal{K}'}(R_{x^k}(t_k\eta^k+t_k^2\tilde{\eta}^k))=x^*.$$ 
	This implies that  $\{{\rm dist}(x^{k+1},x^k)\}_{\mathcal{K}'}\rightarrow 0$. Thus,
	by Proposition \ref{pro4}, we obtain $\{x^k\}\rightarrow x^*$. Furthermore, by the uniqueness of the KKT multiplier, we obtain $\{(x^k,\lambda^{k0})\}\rightarrow (x^*,\lambda^*)$.
\end{proof}

In what follows, we show that the step size $t_k$ always equals $1$
for $k$ sufficiently large, that is the Maratos effect does not occur.
For any $x^k\in\mathcal{M}$, $\xi_{x^k}\in T_{x^k}\mathcal{M}$, and $\lambda\in\mathbb{R}^{m+\ell}$, we define $\pi^k_1:T_{x^k}\mathcal{M}\times\mathbb{R}^{m+\ell}\rightarrow T_{x^k}\mathcal{M}$ and $\pi^k_2:T_{x^k}\mathcal{M}\times\mathbb{R}^{m+\ell}\rightarrow \mathbb{R}^{m+\ell}$ as $\pi^k_1(\xi_{x^k},\lambda)=\xi_{x^k}$ and $\pi^k_2(\xi_{x^k},\lambda)=\lambda$, respectively. It is clear that $\pi^k_1$ and $\pi^k_2$ are linear operators for all $k$. Then we have $\|\pi^k_1\|_{\rm op}=1$ and $\|\pi^k_2\|_{\rm op}=1$ for all $k$. It then follows from \eqref{relationship} that 
\begin{subequations}\label{relationship20}
	\begin{align}
		\label{relationship21}
		&\eta^{k0}=\pi^k_1(\mathcal{A}_k^{-1}(-{\rm grad}F_{\bar{\rho}}(x^k),0)),\,\,\lambda^{k0}=\pi^k_2(\mathcal{A}_k^{-1}(-{\rm grad}F_{\bar{\rho}}(x^k),0)),\\
		\label{relationship22}
		&\eta^{k1}=\eta^{k0}+\pi^k_1(\mathcal{A}_k^{-1}(0_{x^k},v^{k1})),\,\,\lambda^{k1}=\lambda^{k0}+\pi^k_2(\mathcal{A}_k^{-1}(0_{x^k},v^{k1})),\\
		\label{relationship23}
		&\eta^{k2}=\eta^{k1}+\pi^k_1(\mathcal{A}_k^{-1}(0_{x^k}, \Delta v^{k})),\,\,\lambda^{k2}=\lambda^{k1}+\pi^k_2(\mathcal{A}_k^{-1}(0_{x^k},\Delta v^{k})),
	\end{align}
\end{subequations}
where $\Delta v^{k}\in\mathbb{R}^{m+\ell}$ with the components $-\alpha_i^k\|\eta^{k1}\|^\nu$, $i\in \mathcal{L}$. From \eqref{direction}, we further obtain 
\begin{equation}\label{relationship30}
	\eta^{k}=\eta^{k1}+\pi^k_1(\mathcal{A}_k^{-1}(0_{x^k},\theta_k\Delta v^{k})),\,\,\lambda^{k}=\lambda^{k1}+\pi^k_2(\mathcal{A}_k^{-1}(0_{x^k},\theta_k\Delta v^{k})).
\end{equation}


\begin{corollary}\label{co1}
	Suppose that Assumptions \ref{A1}-\ref{A6} hold. Then we have
	\begin{description}
		\item[(i)] $\eta^k\rightarrow 0_{x^*}$, $\eta^{k0}\rightarrow 0_{x^*}$, $\eta^{k1}\rightarrow 0_{x^*}$ as $k\to\infty$;
		\item [(ii)] $\lambda^k\rightarrow\lambda^*$, $\lambda^{k1}\rightarrow\lambda^*$ as $k\to\infty$;
		\item [(iii)] $\mu^k\rightarrow\lambda^*$ as $k\to\infty$.
	\end{description}
	In addition, for $k$ large enough, if follows $\mathcal{L}_k=\mathcal{L}(x^*)$.
\end{corollary}
\begin{lemma}\label{le12}
	Suppose that Assumptions \ref{A1}-\ref{A6} hold. Then we have
	\begin{description}
		\item[(i)] $\alpha_i^k\rightarrow 1$, $\beta_i^k\rightarrow 0$ and $\lambda_i^k\delta^k_i\rightarrow 1/\sqrt{2}$  for $i\in \mathcal{L}(x^*)$ as $k\to\infty$;
		\item[(ii)] $\alpha_i^k\rightarrow 0$ and $\beta_i^k\rightarrow 1$ for $i\notin \mathcal{L}(x^*)$ as $k\to\infty$.
	\end{description} 
\end{lemma}

From \eqref{s12}, we have that $\alpha_i^k\langle{\rm grad}c_i(x^k),\eta^{k0}\rangle-\sqrt{2}\beta_i^k\lambda^{k0}_i=0$ for all $i\in \mathcal{L}$. Considering $i\notin \mathcal{L}(x^*)$, it follows from Lemma \ref{le12} that
\begin{equation}\label{lambdak0}
	\lambda^{k0}_i=o(\|\eta^{k0}\|).
\end{equation}
Therefore, since $\alpha_i^k\in[0,1)$ and $\theta_k\in[0,1]$, then by \eqref{relationship22}, \eqref{relationship30} and the boundedness of $\{\|\mathcal{A}_k^{-1}\|_{\rm op}\}$ and $\{\|\pi^k_1\|_{\rm op}\}$, we obtain
\begin{equation}\label{eta}
	\eta^{k1}=\eta^{k0}+o(\|\eta^{k0}\|^2)\,\,\ \ {\rm and\,\,}\ \ \eta^{k}=\eta^{k0}+o(\|\eta^{k0}\|^2).
\end{equation}

The following lemma shows that the correction direction  $\tilde{\eta}^k$ is of higher-order than the main search direction ${\eta}^k$. This result is the Riemannian counterpart of \cite[Lem. 4.6]{qi2000new}.  For completeness, we provide its proof in Appendix~\ref{AppendixC}.

\begin{lemma}\label{le13}
	Suppose that  Assumptions \ref{A1}-\ref{A6} hold. Then, for $k$ large enough, the direction $\tilde{\eta}^k$ is obtained by solving \eqref{subpro}. Furthermore, we have
	$$\|\tilde{\eta}^k\|=O\left(\max\left\{\|\eta^k\|^2,\max_{i\in \mathcal{L}(x^*)}\bigg\arrowvert\dfrac{\alpha_i^k}{\sqrt{2}\delta^k_i\lambda^k_i}-1\bigg\arrowvert\|\eta^k\|\right\}\right)=o(\|\eta^k\|).$$
\end{lemma}

To demonstrate that the unit step size can be accepted by the arc search of Algorithm \ref{algo1}, we need to make additional assumptions. First, we introduce the concept of second-order retraction. A retraction $R$ is called a second-order retraction if for all $x\in\mathcal{M}$ and for all $\xi\in T_x\mathcal{M}$, $\dfrac{{\rm D}^2}{{\rm d}t^2}R_x(t\xi)\bigg\arrowvert_{t=0}=0$ holds, where the operator $\dfrac{{\rm D}^2}{{\rm d}t^2}$ is the covariant derivative. In fact, a second-order retraction is an approximation of the exponential mapping ${\rm Exp}$; see, e.g., \cite[Sec. 5.5]{absil2008optimization}.
\begin{assumption}\label{A7}
	The retraction $R$ is a second-order retraction.
\end{assumption}

\begin{assumption}\label{A8}
	The sequence $\{\mathcal{H}_k\}$ satisfies $\|(\mathcal{H}_k-{\rm Hess}_{x}L_{\bar{\rho}}(x^k,\lambda^k))\eta^k\|=o(\|\eta^k\|)$.
\end{assumption}

We can now prove that the Maratos effect does not occur in our algorithm.
\begin{theorem}\label{th1}
	Suppose that Assumptions \ref{A1}-\ref{A8} hold. Then for sufficiently large $k$, the step size $t_k$ is always equal to 1.
\end{theorem}
\begin{proof}
	It is sufficient to prove that the inequalities \eqref{linef} and \eqref{lineg} hold when $t_k=1$ and $k$ is sufficiently large. For simplicity, the phrase ``$k$ is sufficiently large" will be implicit throughout the proof. 
	We divide the proof into two parts as follows.
	
	In the first part, we show that the step size $t_k=1$ can be accepted by  \eqref{linef}.    
	Since the retraction $R$ is a second-order retraction, using the second-order Taylor expansion of $F_{\bar{\rho}}$ around $x^k$
	(see, e.g., \cite[Sec. 5.10]{Boumal2023}) and combining it with Lemma \ref{le13}, we obtain   
	\begin{equation}\label{proof-th1-1}
		F_{\bar{\rho}}(R_{x^k}(\eta^k+\tilde{\eta}^k))=F_{\bar{\rho}}(x^k)+\langle{\rm grad}F_{\bar{\rho}}(x^k),\eta^k+\tilde{\eta}^k\rangle+\dfrac{1}{2}\langle\eta^k,{\rm Hess}F_{\bar{\rho}}(x^k)[\eta^k]\rangle+o(\|\eta^k\|^2).
	\end{equation}
	It follows from \eqref{s41}, Lemma \ref{le13}, and  Assumption \ref{A3} that 
	\begin{equation*}
		\begin{aligned}
			&\langle{\rm grad}F_{\bar{\rho}}(x^k),\eta^k\rangle=-\langle\eta^k,\mathcal{H}_k[\eta^k]\rangle-\sum_{i\in \mathcal{L}}\lambda^k_i\langle{\rm grad}c_i(x^k),\eta^k\rangle,\\
			&\langle{\rm grad}F_{\bar{\rho}}(x^k),\tilde{\eta}^k\rangle=-\sum_{i\in \mathcal{L}}\lambda^k_i\langle{\rm grad}c_i(x^k),\tilde{\eta}^k\rangle+o(\|\eta^k\|^2).
		\end{aligned}
	\end{equation*}
	Thus, we can rewrite \eqref{proof-th1-1} as follows:
	\begin{equation}\label{proof-th1-2}
		\begin{aligned}
			F_{\bar{\rho}}(R_{x^k}(\eta^k+\tilde{\eta}^k))&=F_{\bar{\rho}}(x^k)+\dfrac{1}{2}\langle{\rm grad}F_{\bar{\rho}}(x^k),\eta^k\rangle\\
			&+\dfrac{1}{2}\langle\eta^k,({\rm Hess}F_{\bar{\rho}}(x^k)-\mathcal{H}_k)[\eta^k]\rangle	-\dfrac{1}{2}\sum_{i\in \mathcal{L}}\lambda^k_i\langle{\rm grad}c_i(x^k),\eta^k\rangle\\
			&-\sum_{i\in \mathcal{L}}\lambda^k_i\langle{\rm grad}c_i(x^k),\tilde{\eta}^k\rangle+o(\|\eta^k\|^2).
		\end{aligned}
	\end{equation}
	
	For all $i\in \mathcal{L}(x^*)$, it follows that $\tilde{\eta}^k$ satisfies 
	\begin{equation}\label{proof-th1-3}
		c_i(R_{x^k}(\eta^k))+\langle{\rm grad}c_i(x^k),\tilde{\eta}^k\rangle=-\varpi_k.
	\end{equation}
	Expanding \eqref{proof-th1-3} and  using the fact that $\varpi_k=o(\|\eta^k\|^2)$ yield 
	\begin{equation}\label{proof-th1-4}
		c_i(x^k)+\langle{\rm grad}c_i(x^k),\eta^k\rangle+\dfrac{1}{2}\langle\eta^k,{\rm Hess}c_i(x^k)[\eta^k]\rangle+\langle{\rm grad}c_i(x^k),\tilde{\eta}^k\rangle=o(\|\eta^k\|^2).
	\end{equation}
	Multiplying \eqref{proof-th1-4} by $\lambda_i^k$, summing over $i\in \mathcal{L}(x^*)$, and noting that $\alpha_i^k\neq 0$, $\beta^k_i=-\delta^k_ic_i(x^k)$ and  \eqref{s42}, we obtain
	\begin{equation}\label{proof-th1-5}
		\begin{aligned}
			&-\dfrac{1}{2}\sum_{i\in \mathcal{L}(x^*)}\lambda^k_i\langle{\rm grad}c_i(x^k),\eta^k\rangle-\sum_{i\in \mathcal{L}(x^*)}\lambda^k_i\langle{\rm grad}c_i(x^k),\tilde{\eta}^k\rangle\\
			=&\sum_{i\in \mathcal{L}(x^*)}\left(1-\dfrac{\sqrt{2}\delta^k_i\lambda^k_i}{2\alpha^k_i}\right)\lambda^k_ic_i(x^k)\\
			&+\dfrac{1}{2}\sum_{i\in \mathcal{L}(x^*)}\lambda^k_i\langle\eta^k,{\rm Hess}c_i(x^k)[\eta^k]\rangle+o(\|\eta^k\|^2).
		\end{aligned}
	\end{equation}
	
	For $i\notin \mathcal{L}(x^*)$, if $\alpha_i^k=0$, then by \eqref{s42}, we know that $\lambda_i^k=0$. Now, we consider the case that $\alpha_i^k\neq 0$. It follows from \eqref{s42} and $\lambda_i^k\rightarrow 0$ as $k\to\infty$ that
	\begin{equation}\label{proof-th1-6}
		\lambda^k_i\langle{\rm grad}c_i(x^k),\eta^k\rangle=\dfrac{\sqrt{2}\beta_i^k(\lambda_i^k)^2}{\alpha_i^k}-\lambda^k_i\langle\eta^k,{\rm Hess}c_i(x^k)[\eta^k]\rangle+o(\|\eta^k\|^2).
	\end{equation}
	Summing over $i\in J^k_\alpha:=\{i\in \mathcal{L}\mid \alpha_i^k\neq 0, i\notin \mathcal{L}(x^*)\}$, and using $\beta^k_i=-\delta^k_ic_i(x^k)$, we get
	\begin{equation}\label{proof-th1-7}
		\begin{aligned}
			-\sum_{i\in J^k_\alpha}\lambda^k_i\langle{\rm grad}c_i(x^k),\eta^k\rangle&=\sum_{i\in J^k_\alpha}\dfrac{\sqrt{2}\delta^k_i(\lambda^k_i)^2}{\alpha_i^k}c_i(x^k)\\
			&+\sum_{i\in J^k_\alpha}\lambda^k_i\langle\eta^k,{\rm Hess}c_i(x^k)[\eta^k]\rangle+o(\|\eta^k\|^2).
		\end{aligned}
	\end{equation}
	On the other hand, it follows from \eqref{relationship20}, \eqref{relationship30}, \eqref{lambdak0}, \eqref{eta} and Lemma \ref{le13} that 
	\begin{equation}\label{proof-th1-8}
		\lambda^k_i\langle{\rm grad}c_i(x^k),\tilde{\eta}^k\rangle=o(\|\eta^k\|^2),\quad \forall i\notin \mathcal{L}(x^*).
	\end{equation}
	Then by \eqref{proof-th1-5}, \eqref{proof-th1-7},  \eqref{proof-th1-8} and  the fact that $\lambda_i^k=0$ whenever $i\notin \mathcal{L}(x^*)$ with $\alpha_i^k=0$, we obtain
	\begin{equation}\label{proof-th1-9}
		\begin{aligned}
			&-\dfrac{1}{2}\sum_{i\in \mathcal{L}}\lambda^k_i\langle{\rm grad}c_i(x^k),\eta^k\rangle-\sum_{i\in \mathcal{L}}\lambda^k_i\langle{\rm grad}c_i(x^k),\tilde{\eta}^k\rangle\\
			=&\sum_{i\in \mathcal{L}(x^*)}\left(1-\dfrac{\sqrt{2}\delta^k_i\lambda^k_i}{2\alpha^k_i}\right)\lambda^k_ic_i(x^k)+\sum_{i\in J^k_\alpha}\dfrac{\sqrt{2}\delta^k_i(\lambda^k_i)^2}{2\alpha_i^k}c_i(x^k)\\
			&+\dfrac{1}{2}\sum_{i\in \mathcal{L}}\lambda^k_i\langle\eta^k,{\rm Hess}c_i(x^k)[\eta^k]\rangle+o(\|\eta^k\|^2).
		\end{aligned}
	\end{equation}
	This together with \eqref{proof-th1-2} shows that
	\begin{equation}\label{proof-th1-10}
		\begin{aligned}
			F_{\bar{\rho}}(R_{x^k}(\eta^k+\tilde{\eta}^k))&=F_{\bar{\rho}}(x^k)+\dfrac{1}{2}\langle{\rm grad}F_{\bar{\rho}}(x^k),\eta^k\rangle\\
			&+\dfrac{1}{2}\left\langle\eta^k,\left({\rm Hess}F_{\bar{\rho}}(x^k)+\sum_{i\in \mathcal{L}}\lambda_i^k{\rm Hess}c_i(x^k)-\mathcal{H}_k\right)[\eta^k]\right\rangle\\
			&+\sum_{i\in \mathcal{L}(x^*)}\left(1-\dfrac{\sqrt{2}\delta^k_i\lambda^k_i}{2\alpha^k_i}\right)\lambda^k_ic_i(x^k)\\
			&+\sum_{i\in J^k_\alpha}\dfrac{\sqrt{2}\delta^k_i(\lambda^k_i)^2}{2\alpha_i^k}c_i(x^k)+o(\|\eta^k\|^2).
		\end{aligned}
	\end{equation}
	
	From Lemma \ref{le12}, we know that
	$$\sum_{i\in \mathcal{L}(x^*)}\left(1-\dfrac{\sqrt{2}\delta^k_i\lambda^k_i}{2\alpha^k_i}\right)\lambda^k_ic_i(x^k)\leq 0.$$
	And it is obvious that 
	$$\sum_{i\in J^k_\alpha}\dfrac{\sqrt{2}\delta^k_i(\lambda^k_i)^2}{2\alpha_i^k}c_i(x^k)\leq 0.$$
	Then, it follows from \eqref{proof-th1-10} that
	\begin{equation*}\label{proof-th1-11}
		\begin{aligned}
			F_{\bar{\rho}}(R_{x^k}(\eta^k+\tilde{\eta}^k))&\leq F_{\bar{\rho}}(x^k)+\dfrac{1}{2}\langle{\rm grad}F_{\bar{\rho}}(x^k),\eta^k\rangle\\
			&+\dfrac{1}{2}\left\langle\eta^k,\left({\rm Hess}F_{\bar{\rho}}(x^k)+\sum_{i\in \mathcal{L}}\lambda^k_i{\rm Hess}c_i(x^k)-\mathcal{H}_k\right)[\eta^k]\right\rangle+o(\|\eta^k\|^2)\\
			&\leq F_{\bar{\rho}}(x^k)+\dfrac{1}{2}\langle{\rm grad}F_{\bar{\rho}}(x^k),\eta^k\rangle\\
			&+\dfrac{1}{2}\|(\mathcal{H}_k-{\rm Hess}_{x^k}L_{\bar{\rho}}(x^k,\lambda^k))\eta^k\|\|\eta^k\|+o(\|\eta^k\|^2).
		\end{aligned}
	\end{equation*}
	Thus,  by the boundedness of $\{\eta^k\}$ and Assumption \ref{A8}, we have
	\begin{equation*}
		F_{\bar{\rho}}(R_{x^k}(\eta^k+\tilde{\eta}^k))\leq F_{\bar{\rho}}(x^k)+\dfrac{1}{2}\langle{\rm grad}F_{\bar{\rho}}(x^k),\eta^k\rangle+o(\|\eta^k\|^2).
	\end{equation*}
	This along with \eqref{dk01}, \eqref{eta}, Lemma \ref{le3} and Assumption \ref{A3} implies that
	\begin{equation*}
		F_{\bar{\rho}}(R_{x^k}(\eta^k+\tilde{\eta}^k))-F_{\bar{\rho}}(x^k)-\sigma\langle{\rm grad}F_{\bar{\rho}}(x^k),\eta^k\rangle\leq (\sigma-\dfrac{1}{2})\tau a_1\|\eta^{k0}\|^2+o(\|\eta^{k0}\|^2),
	\end{equation*}
	which completes the proof of the first part due to $\sigma\in(0,1/2)$. 
	
	For the second part, we show that the step size $t_k=1$ can be accepted by \eqref{lineg}. For $i\notin \mathcal{L}(x^*)$, by the continuity of $c_i(x)$ and Corollary \ref{co1} (i), it is clear that the step size $t_k=1$ can be accepted by such $c_i$. For $i\in \mathcal{L}(x^*)$, by expanding $c_i\circ R_{x^k}$ in the linear space $T_{x^k}\mathcal{M}$ around $\eta^k$, we have
	\begin{equation*}
		\begin{aligned}
			c_i(R_{x^k}(\eta^k+\tilde{\eta}^k))
			=&c_i\circ R_{x^k}(\eta^k)+\langle{\rm grad}(c_i\circ R_{x^k})(\eta^k),\tilde{\eta}^k\rangle+O(\|\tilde{\eta}^k\|^2),
		\end{aligned}
	\end{equation*}
	where ${\rm grad}(c_i\circ R_{x^k})(\eta^k)$ denotes the gradient of $c_i\circ R_{x^k}$ at $\eta^k$. Note that ${\rm grad}(c_i\circ R_{x^k})(0_{x^k})={\rm grad}c_i(x^k)$ (see \cite[Sec. 4.1]{absil2008optimization}). Since $c_i(x)$ is twice continuously differentiable and $R$ is smooth, 
	it follows from Lemma \ref{le13} that
	\begin{equation*}
		c_i(R_{x^k}(\eta^k+\tilde{\eta}^k))
		=c_i(R_{x^k}(\eta^k))+\langle{\rm grad}c_i(x^k),\tilde{\eta}^k\rangle+O(\|\eta^k\|\|\tilde{\eta}^k\|).
	\end{equation*}
	This together with \eqref{subpro} shows that
	$$c_i(R_{x^k}(\eta^k+\tilde{\eta}^k))=-\varpi_k+O(\|\eta^k\|\|\tilde{\eta}^k\|).$$
	Then, it follows from Lemma \ref{le13} and the definition of $\varpi_k$ that
	$$c_i(R_{x^k}(\eta^k+\tilde{\eta}^k))=-\varpi_k+o(\varpi_k)<0.$$
	Thus, the proof is completed.
\end{proof} 

For the purpose of applying the results in \cite{qi2000new}, we shall convert the iterative format to Euclidean space below for $k$ large enough.
From Theorem \ref{th1}, Lemma \ref{le13} and \eqref{eta}, we can conclude that 
$x^{k+1}=R_{x^k}(\eta^k+\tilde{\eta}^k)=R_{x^k}(\eta^k+o(\|\eta^k\|))=R_{x^k}(\eta^{k0}+o(\|\eta^{k0}\|))$ for $k$ large enough.
We further denote $H_k={\rm D}\varphi(x^k)\mathcal{H}_k{\rm D}\varphi^{-1}(\hat{x}^k)$ and $\hat{R}_{\hat{x}}(\hat{\xi})=\varphi(R_x{\xi})$. Thus $H_k$ is an $n\times n$ matrix, and $\hat{R}_{\hat{x}}$ is a mapping from $\mathbb{R}^n$ to $\mathbb{R}^n$ satisfying ${\rm D}\hat{R}_{\hat{x}}(0)={\rm id}$ (see, e.g., \cite[Thm. 6.3.2]{absil2008optimization}). 

Since $x^k\rightarrow x^*$ as $k\to\infty$, it follows that $x^k\in B(x^*,r_1)=\{x\in U\mid {\rm dist}(x,x^*)\leq r_1\}$ for all sufficiently large $k$ and some $r_1>0$. Thus $\{\|{\rm D}\varphi(x^k)\|_{\rm op}\}$ is bounded on $B(x^*,r_1)$ by the smoothness of $\varphi$. On the other hand, by the continuity of $G_{\hat{x}^k}$ and the fact that $G_{\hat{x}^*}=I_n$, we have
\begin{equation}\label{limeta}
	\lim_{k\rightarrow\infty}\dfrac{\|\eta^{k0}\|}{\|\hat{\eta}^{k0}\|}=1.
\end{equation}
This implies
$${\rm D}\varphi(x^k)[\eta^{k0}+o(\|\eta^{k0}\|)]=\hat{\eta}^{k0}+o(\|\hat{\eta}^{k0}\|),$$
and therefore
\begin{equation}\label{exk+11}
	\hat{x}^{k+1}=\hat{R}_{\hat{x}}(\hat{\eta}^{k0}+o(\|\hat{\eta}^{k0}\|)).
\end{equation}
%

The following lemma shows that Assumption \ref{A3} and \ref{A8} can be locally transformed into the Euclidean setting.

\begin{lemma}\label{le16}
	Suppose that Assumptions \ref{A3} and \ref{A8} hold. Then for all $k\in\mathbb{N}$ large enough, we have 
	\begin{description}
		\item[(i)] the matrix $G_{\hat{x}^k}H_k$ is symmetric, and there exist two constants $\bar{a}_2>\bar{a}_1>0$ such that $\bar{a}_1\|v\|^2\leq v^{\top}G_{\hat{x}^k}H_kv\leq\bar{a}_2\|v\|^2$ for any $v\in\mathbb{R}^n$.
		\item [(ii)] $\|(G_{\hat{x}^k}H_k-\nabla_{\hat{x}}\hat{L}(\hat{x}^*,\lambda^*))\hat{\eta}^{k0}\|=o(\|\hat{\eta}^{k0}\|)$.
	\end{description}
\end{lemma}
\begin{proof}
	(i) For any $v_1,v_2\in\mathbb{R}^n$, we have $v_1^{\top}H_k^{\top}G_{\hat{x}^k}v_2=v_1^{\top}G_{\hat{x}^k}H_kv_2$. This implies that $G_{\hat{x}^k}H_k$ is symmetric. In addition, for any $v\in\mathbb{R}^n$, we have 
	$$v^{\top}G_{\hat{x}^k}H_kv=\left\langle{\rm D}\varphi^{-1}(\hat{x}^k)[v],\mathcal{H}_k[{\rm D}\varphi^{-1}(\hat{x}^k)[v]]\right\rangle,$$ 
	for all $k$ large enough. Then by Assumption \ref{A3}, we have
	$$a_1\|{\rm D}\varphi^{-1}(\hat{x}^k)[v]\|^2\leq v^{\top}G_{\hat{x}^k}H_kv\leq a_2\|{\rm D}\varphi^{-1}(\hat{x}^k)[v]\|^2.$$
	We have already obtained the boundedness of $\{\|{\rm D}\varphi(x^k)\|_{\rm op}\}$ on $B(x^*,r_1)$. Similarly, we can also obtain that the sequence $\{\|{\rm D}\varphi^{-1}(\hat{x}^k)\|_{\rm op}\}$ (with $k$ being large enough) is bounded. Thus, there exist two constants $M_1>0$ and $M_2>0$ such that $\|{\rm D}\varphi(x^k)\|_{\rm op}\leq M_1$ and $\|{\rm D}\varphi^{-1}(\hat{x}^k)\|_{\rm op}\leq M_2$ for all $k$ large enough. Then we have 
	\begin{equation*}
		\|v\|=\|{\rm D}\varphi(x^k)\circ {\rm D}\varphi^{-1}(\hat{x}^k)[v]\|\leq\|{\rm D}\varphi(x^k)\|_{\rm op}\|{\rm D}\varphi^{-1}(\hat{x}^k)[v]\|\leq M_1\|{\rm D}\varphi^{-1}(\hat{x}^k)[v]\|.
	\end{equation*}
	By setting $\bar{a}_1=a_1/M^2_1$ and $\bar{a}_2=a_2 M^2_2$, we obtain part (i).
	
	(ii) It follows from \eqref{eta} and Assumption \ref{A8} that
	\begin{equation}\label{proof-le14-1}
		\|(\mathcal{H}_k-{\rm Hess}_{x}L(x^k,\lambda^k))[\eta^{k0}]\|=o(\|\eta^{k0}\|).
	\end{equation}
	On the other hand, we have
	\begin{equation}\label{proof-le14-2}
		\begin{aligned}
			&\dfrac{\|(G_{\hat{x}^k}H_k-\nabla_{\hat{x}}\hat{L}(\hat{x}^*,\lambda^*))[\hat{\eta}^{k0}]\|}{\|\hat{\eta}^{k0}\|}\\
			\leq&\dfrac{\|(H_k-{\rm D}\varphi(x^k){\rm Hess}_{x}L(x^k,\lambda^k){\rm D}\varphi^{-1}(\hat{x}^k))[\hat{\eta}^{k0}]\|}{\|\hat{\eta}^{k0}\|}\\
			+&\dfrac{\|({\rm D}\varphi(x^k){\rm Hess}_{x}L(x^k,\lambda^k){\rm D}\varphi^{-1}(\hat{x}^k)-\nabla_{\hat{x}}\hat{L}(\hat{x}^*,\lambda^*))[\hat{\eta}^{k0}]\|}{\|\hat{\eta}^{k0}\|}\\
			+&\dfrac{\|(G_{\hat{x}^k}H_k-H_k)[\hat{\eta}^{k0}]\|}{\|\hat{\eta}^{k0}\|}\\
			\leq&\dfrac{\|{\rm D}\varphi(x^k)\|_{\rm op}\|{\eta}^{k0}\|\|(\mathcal{H}_k-{\rm Hess}_{x}L(x^k,\lambda^k)({x}^k))[\hat{\eta}^{k0}]\|}{\|\hat{\eta}^{k0}\|\|\eta^{k0}\|}\\
			+&\|{\rm D}\varphi(x^k){\rm Hess}_{x}L(x^k,\lambda^k){\rm D}\varphi^{-1}(\hat{x}^k)-\nabla_{\hat{x}}\hat{L}(\hat{x}^*,\lambda^*)\|\\
			+&\|G_{\hat{x}^k}H_k-H_k\|.
		\end{aligned}
	\end{equation}
	Since $\{\|{\rm D}\varphi(x^k)\|_{\rm op}\}$ is bounded, we have $G_{\hat{x}^k}\rightarrow I_n$. 
	This together with part (i) of this lemma, \eqref{limeta}, \eqref{proof-le14-1}, \eqref{proof-le14-2} and Lemma \ref{le10} shows part (ii).
\end{proof}

Finally, we analyze the local convergence rate of Algorithm \ref{algo1}.
It follows from \eqref{s1} and $\delta_i^k=-\beta_i^k/c_i(x^k)$ that 
\begin{subequations}\label{s5}
	\begin{align}
		\label{s51}
		&\mathcal{H}_k[\eta^{k0}]+\sum_{i\in \mathcal{L}}\lambda_i^{k0}{\rm grad}c_i(x^k)=-{\rm grad}F_{\bar{\rho}}(x^k),\\
		\label{s52}
		&\alpha_i^k\langle{\rm grad}c_i(x^k),\eta^{k0}\rangle+\sqrt{2}\lambda_i^{k0}\delta^k_ic_i(x^k)=0,\quad i\in \mathcal{L}.
	\end{align}
\end{subequations}
Then by \eqref{lambdak0} and \eqref{s51}, we have
\begin{equation}\label{dk011}
	\mathcal{H}_k[\eta^{k0}]+\sum_{i\in \mathcal{L}(x^*)}\lambda_i^{k0}{\rm grad}c_i(x^k)=-{\rm grad}F_{\bar{\rho}}(x^k)+o(\|\eta^{k0}\|).
\end{equation}
On the other hand, for $i\in \mathcal{L}(x^*)$, it follows from \eqref{relationship20}, \eqref{relationship30} and Lemma \ref{le12} that $\sqrt{2}\lambda_i^{k0}\delta^k_i\rightarrow 1$ as $k\to\infty$. Therefore, for $k$ large enough, we have $\sqrt{2}\lambda_i^{k0}\delta^k_i\neq 0$. This together with \eqref{s52} implies
$$\dfrac{\alpha_i^k}{\sqrt{2}\lambda_i^{k0}\delta^k_i}\langle{\rm grad}c_i(x^k),\eta^{k0}\rangle+c_i(x^k)=0,\,\,\forall i\in \mathcal{L}(x^*).$$
Note that $\dfrac{\alpha_i^k}{\sqrt{2}\lambda_i^{k0}\delta^k_i}\rightarrow 1$ as $k\to\infty$ by Lemma \ref{le12}. This leads to the conclusion that
\begin{equation}\label{dk012}
	\langle{\rm grad}c_i(x^k),\eta^{k0}\rangle+c_i(x^k)=o(\|\eta^{k0}\|), \ 
	\,\,\forall i\in \mathcal{L}(x^*).
\end{equation}

It follows from \eqref{limeta}, \eqref{dk011}, \eqref{dk012}, $\widehat{{\rm grad}c_i}(\hat{x}^k)=G_{x^k}^{-1}\nabla\hat{c}_i(\hat{x}^k)$ and $\widehat{{\rm grad}F_{\bar{\rho}}}(\hat{x}^k)=G_{x^k}^{-1}\nabla\hat{F}_{\bar{\rho}}(\hat{x}^k)$ that
\begin{equation*}
	\begin{aligned}
		&G_{\hat{x}^k}{H}_k\hat{\eta}^{k0}+\sum_{i\in L(x^*)}\lambda_i^{k0}\nabla\hat{c}_i(\hat{x}^k)=-\nabla\hat{F}_{\bar{\rho}}(\hat{x}^k)+o(\|\hat{\eta}^{k0}\|),\\
		&\langle\nabla\hat{c}_i(\hat{x}^k),\hat{\eta}^{k0}\rangle+\hat{c}_i(\hat{x}^k)=o(\|\hat{\eta}^{k0}\|),\quad i\in \mathcal{L}(x^*).
	\end{aligned}
\end{equation*}
That is
\begin{equation*}       
	\left(               
	\begin{array}{cc}   
		G_{\hat{x}^k}H_k & N_k\\
		N_k^{\top} & 0 \\ 
	\end{array}
	\right)
	\left(\begin{array}{c}   
		\hat{\eta}^{k0}\\ 
		\lambda^{k0}_{\mathcal{L}(x^*)}\\ 
	\end{array}\right) 
	= \left(\begin{array}{c}   
		-\nabla\hat{F}_{\bar{\rho}}(\hat{x}^k)\\ 
		-\hat{c}_{\mathcal{L}(x^*)}(\hat{x}^k)\\ 
	\end{array}\right)+ o(\|\hat{\eta}^{k0}\|),  
\end{equation*}
where $N_k=[\nabla\hat{c}_i(\hat{x}^k),i\in \mathcal{L}(x^*)]$, $\lambda^{k0}_{\mathcal{L}(x^*)}=(\lambda^{k0}_i,i\in \mathcal{L}(x^*))^{\top}$ and $\hat{c}_{\mathcal{L}(x^*)}=(\hat{c}_i,i\in \mathcal{L}(\hat{x}^*)=\mathcal{L}(x^*))^{\top}$.

From the above analysis, we can obtain the following lemma, whose proof is similar to that of \cite[Thm. 4.9]{qi2000new}.
\begin{lemma}\label{le14}
	Suppose that Assumptions \ref{A1}-\ref{A8} hold. Then we have
	$$\|\hat{x}^k+\hat{\eta}^{k0}+o(\|\hat{\eta}^{k0}\|)-\hat{x}^*\|=o\|\hat{x}^k-\hat{x}^*\|.$$
\end{lemma}

At the end of this section, we establish the superlinear convergence result for Algorithm  \ref{algo1}.
\begin{theorem}
	Suppose that Assumptions \ref{A1}-\ref{A8} hold. Then the sequence $\{x^k\}$ converges Q-superlinearly, that is
	$$\|\hat{x}^{k+1}-\hat{x}^*\|=o(\|\hat{x}^k-\hat{x}^*\|).$$
\end{theorem}
\begin{proof}
	Similar to Lemma 4.8 and equation (4.9) of \cite{qi2000new}, we have
	\begin{equation}\label{proof-th7-1}
		\|\hat{\eta}^{k0}+o(\|\hat{\eta}^{k0}\|)\|^2=o(\|\hat{x}^k-\hat{x}^*\|).
	\end{equation}
	In addition, from \cite[Thm.6.3.2]{absil2008optimization}, we obtain
	\begin{equation}\label{proof-th7-2}
		\|\hat{R}_{\hat{x}}(\hat{\eta}^{k0}+o(\|\hat{\eta}^{k0}\|))-(\hat{x}^k+\hat{\eta}^{k0}+o(\|\hat{\eta}^{k0})\|\leq c\|\hat{\eta}^{k0}+o(\|\hat{\eta}^{k0}\|)\|^2,
	\end{equation}
	where $c>0$ is a constant. This together with \eqref{proof-th7-2} and Lemma \ref{le14} shows that
	\begin{equation*}
		\begin{aligned}
			\|\hat{x}^{k+1}-\hat{x}^*\|\leq&\|\hat{R}_{\hat{x}}(\hat{\eta}^{k0}+o(\|\hat{\eta}^{k0}\|))-(\hat{x}^k	+\hat{\eta}^{k0}+o(\|\hat{\eta}^{k0})\|\\
			&+\|\hat{x}^k+\hat{\eta}^{k0}+o(\|\hat{\eta}^{k0}\|)-\hat{x}^*\|\\
			\leq& c\|\hat{\eta}^{k0}+o(\|\hat{\eta}^{k0}\|)\|^2+o(\|\hat{x}^k-\hat{x}^*\|).
		\end{aligned}
	\end{equation*}
	Therefore, by combining with \eqref{proof-th7-1}, we complete the proof.
\end{proof}

\section{Numerical experiments} \label{Sec:numerical}

The proposed algorithm is compared with state-of-the-art algorithms on problems from the nonnegative low-rank matrix completion in~\cite{obara2022sequential} and the nonnegative principal component analysis in~\cite{jiang2023exact}.
All experiments are implemented in MATLAB R2021a using Manopt 8.0 \cite{boumal2014manopt}, on a laptop equipped with an Intel(R) Core(TM) i9-14900HX CPU @2.20GHz and 16GB of memory. 

\subsection{Problem setting}

\textbf{Nonnegative low-rank matrix completion.} Recovering a matrix from a small number of observed entries constitutes a low-rank matrix completion problem (see e.g., \cite{guglielmi2020efficient,vandereycken2013low}). Let $\mathcal{M}_r=\{X\in\mathbb{R}^{d\times s}\mid {\rm rank}(X)=r\}$ be a fixed-rank manifold and $\mathcal{N}$ be a subset of $\{1,\cdots,d\}\times \{1,\cdots,s\}$. The nonnegative low-rank matrix completion problem in~\cite{obara2022sequential}
is given by 
\begin{equation}\label{low-rank}
	\begin{aligned}
		&\min_{X\in\mathcal{M}_r} \dfrac{1}{2}\|P_{\mathcal{J}\setminus \mathcal{G}}(X-A)\|^2_{\rm F}\\
		&{\rm \quad\,\,s.t.}\,\,\, X_{ij}\geq 0 \quad\quad{\rm for\,\, all}\,\,(i,j)\in \mathcal{N}\setminus \mathcal{J},\\
		&\quad\quad\,\,\,\,\,\,\,X_{ij}= A_{ij}\quad {\rm for\,\, all}\,\,(i,j)\in \mathcal{G},
	\end{aligned}
\end{equation}
where $\mathcal{G}\subseteq \mathcal{J}\subseteq \mathcal{N}$ and $A\in\mathbb{R}^{d\times s}$ is a matrix whose entries indexed by $\mathcal{J}$ are known a priori. For a set $\mathcal{S}\subseteq \mathcal{N}$ and a matrix $C\in \mathbb{R}^{d\times s}$, $P_{\mathcal{S}}(C)$ denotes a matrix whose $(i,j)$-th entry is $C_{ij}$ if $(i,j)\in \mathcal{S}$ and 0 otherwise. {We follow the same approach as \cite{obara2022sequential} to construct the matrix $A\in\mathbb{R}^{d\times s}$. Specifically, We repeatedly draw random matrices $T\in\mathbb{R}^{d\times r}$ and $V\in\mathbb{R}^{r\times s}$ with i.i.d. entries uniformly distributed on $(0,1)$ until $\operatorname{rank}(TV)=r$, and then set $A=TV$. The sets $\mathcal{N}$, $\mathcal{J}$ and $\mathcal{G}$ are randomly generated using the \texttt{randperm} function, with $\lvert \mathcal{N}\rvert=\lceil0.8sd\rceil$, $\lvert \mathcal{J}\rvert=\lceil \lvert \mathcal{N}\rvert/4\rceil$ and $\lvert \mathcal{G}\rvert=\lceil \lvert \mathcal{J}\rvert/4\rceil$, where $\lceil\cdot\rceil$ denotes the ceiling function. }

\textbf{Nonnegative principal component analysis (PCA).} In~\cite{jiang2023exact}, the nonnegative PCA is formulated as a Riemannian constrained optimization problem 
\begin{equation}\label{non-pca}
	\begin{aligned}
		&\min_{X\in {\rm Ob}(d,s)} {\rm Tr}(X^{\top}CX)+\Upsilon (\zeta_q(X)+\chi)^p\\
		&{\rm \qquad\,\,s.t.}\,\,\, X_{ij}\geq 0\quad{\rm for\,\, all}\,\, (i,j)\in \mathcal{N},
	\end{aligned}
\end{equation}
where ${\rm Ob}(d,s)=\{X\in\mathbb{R}^{d\times s}\mid {\rm diag}(X^{\top}X)=I_s\}$ denote the oblique manifold, $C=-AA^{\top}$ with $A\in \mathbb{R}^{d\times l}$; $\zeta_q(X)=\|XV\|^q_{\rm F}-1$ with $V\in \mathbb{R}^{s\times r}$ and $r$ being arbitrary as long as $\|V\|_{\rm F}=1$ and $VV^{\top}$ is entrywise positive; $\Upsilon>0$ is the penalty parameter; and $p,q>0$ and $\chi\geq 0$ are the model parameters.
In this paper, we set $l=s$, $r=1$, $p=1$, $q=2$, and $\chi=0$, $\Upsilon=0.5$, and $V=(1/\sqrt{s}) e\in \mathbb{R}^s$, where $e$ denotes a vector with all entries being one. {We generate $A$ with i.i.d. entries uniformly distributed on $(0,1)$.}


\subsection{Implementations}

Algorithm~\ref{algo1} (RQO-free) is compared with the Riemannian augmented Lagrangian method in \cite{liu2020simple} (RALM), the Riemannian exact penalty method~\cite{liu2020simple} with smoothing functions being linear-quadratic and pseudo-Huber (REPM-LQH) and being log-sum exp (REPM-LSE), and the Riemannian sequential quadratic optimization in~\cite{obara2022sequential} (RSQO). The implementations of the benchmark methods RALM, REPMs, and RSQP are adopted from~\cite{obara2022sequential}. All source code is available at \url{https://github.com/haohe2904/RQO-free.git}.








The KKT residual in~\cite{obara2022sequential} for Riemannian constrained optimization is defined by
{\fontsize{8pt}{10pt}\selectfont
	$$\sqrt{\|{\rm grad} L(x,\lambda)\|^2+\sum_{i\in\mathcal{I}}\left(\max(0,-\lambda_i)^2+\max(0,c_i(x))^2+(\lambda_ic_i(x))^2\right)+\sum_{i\in\mathcal{E}}\lvert c_i(x)\rvert^2}+\iota(x),$$}
where $L$ denotes the Riemannian Lagrangian function of problem \eqref{Rpro}, $\lambda\in\mathbb{R}^{m+\ell}$ is a Lagrangian multiplier vector,and $\iota$ is defined by
\begin{align*}
	\hbox{Oblique manifold:} \qquad \iota(x)=&\|{\rm diag}(x^{\top}x)-I_s\|_{\rm F}, \\
	\hbox{Fixed-rank manifold:} \qquad \iota(x) =& 
	\begin{cases}
		0 & \hbox{if $x$ is on the manifold;} \\
		\infty & \hbox{ otherwise. }
	\end{cases}
\end{align*}
The KKT residual is used later to evaluate the quality of the compared methods.


	For the subproblems in the RQO-free method, we use the technique in~\cite{huang2017intrinsic} to transform the linear systems \eqref{s1}, \eqref{s2}, \eqref{s3}, and the subproblem \eqref{subpro} into their Euclidean forms. The Euclidean linear systems of \eqref{s1}, \eqref{s2}, and \eqref{s3} are solved by using \textbf{the LU} decomposition and the Euclidean form of the subproblem~\eqref{subpro} is solved by invoking \texttt{quadprog} (a MATLAB solver for quadratic optimization problems). Note that the subproblem~\eqref{subpro} is used to avoid the Maratos effect only when the KKT residual is smaller than $10^{-5}$.

All tested algorithms terminate if the number of iterations reaches $N_{\max}$ or the computational time exceeds $T_{\max}$, where $N_{\max}$ and $T_{\max}$ are given later. The RALM and REPMs further terminate if the algorithms do not update any parameters, i.e., they keep producing the same point from that point onward. The default parameters of RALM, REPMs, and RSQP are used. For the RQO-free method, the parameters are set as $\nu=2.3$, $\tau=0.75$, $\varrho=2.4$, $\kappa=0.55$, $\sigma=0.45$, $\varsigma=0.5$, $\tilde{\rho}=1.5$, $\rho^0=2$, $r_1=r_2=r_3=0.5$, $\bar{\mu}=50$, and $\mu_i^0=0.1$ for all $i\in \mathcal{L}$.



For Problem \eqref{low-rank}, since a randomly generated point is often not strictly feasible, we follow \cite{obara2022sequential} to compute a candidate $X^c$: solve $\min_{X\in\mathcal{M}_r} \frac12\|P_{\mathcal{J}\setminus \mathcal{G}}(X-\frac12 A)\|_{\rm F}^2$ subject to $X_{ij}\geq0$ for $(i,j)\in\mathcal{N}\setminus\mathcal{J}$ and $X_{ij}=\frac12 A_{ij}$ for $(i,j)\in\mathcal{G}$ using REPM-LQH until the KKT residual falls below $10^{-2}$. If $X^c$ is strictly feasible, set $X^0=X^c$; otherwise, repeat. We use $\frac12 A_{ij}$ instead of $A_{ij}$ because, given $A_{ij}>0$, if $X_{ij}\approx\frac12 A_{ij}$, then $X_{ij}<A_{ij}$ is likely to hold, making it easier to obtain a strictly feasible point. 
For Problem \eqref{non-pca}, we similarly obtain $X^c$ via REPM-LQH but terminate at KKT residual $10^{-1}$. If $X^c$ is feasible, set $X^0=X^c$; otherwise set $X^0=\operatorname{abs}(X^c)$, where $\operatorname{abs}(\cdot)$ denotes the entrywise absolute value of a matrix.

\subsection{Results analysis}

{For nonnegative low-rank matrix
	completion problem \eqref{low-rank}}. A typical run demonstrating the superlinear convergence rate of ROQ-free is shown in Figure~\ref{fig1}, which aligns with our theoretical results.
Figure~\ref{fig2} plots the KKT residual of all compared methods.
We can see that the RQO-free method exhibits a significant advantage over the other methods in terms of computational efficiency and accuracy. Although RALM and REPMs require the fewest outer iterations, the high computational cost per iteration negates the advantage in total computation time. 
Moreover, the RQO-free method yields a much more accurate solution than other methods. Such observations are verified by Table~\ref{tab_rank1}.

\begin{figure}[h]
	\centering    \includegraphics[width=1.0\linewidth]{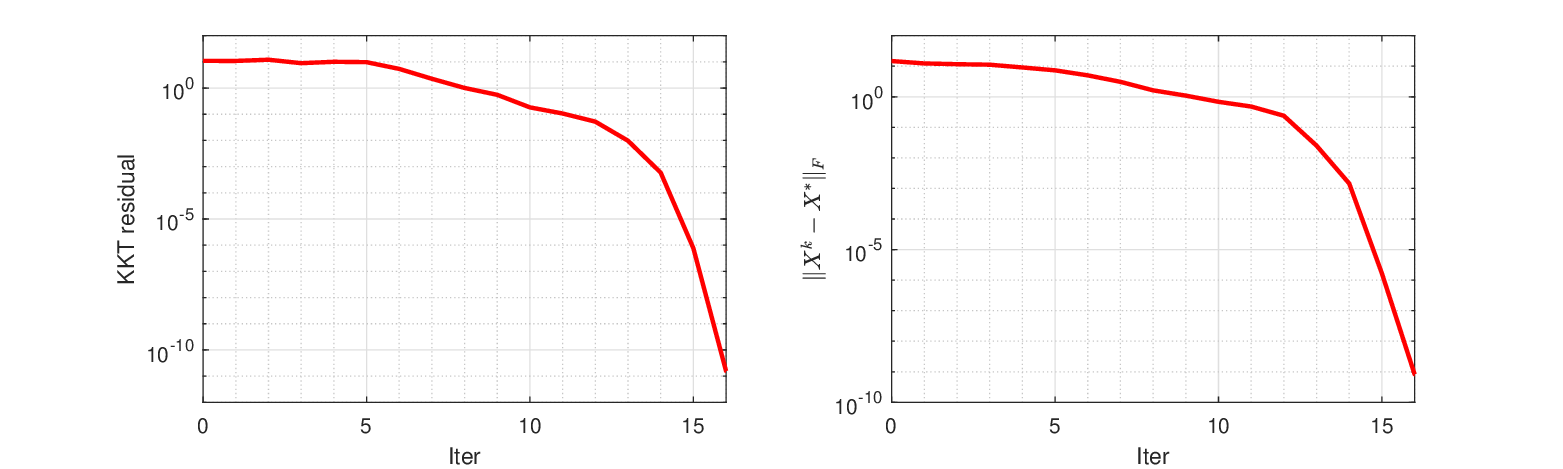}
	\caption{The convergence behavior of the KKT residual of Algorithm \ref{algo1} on problem \eqref{low-rank} with $(d,s,r)=(10,15,4)$. $X^*$ is a high-accuracy solution with KKT residual on the magnitude of $10^{-14}$.}
	\label{fig1}
\end{figure}

\begin{figure}[h]
	\centering    \includegraphics[width=1.0\linewidth]{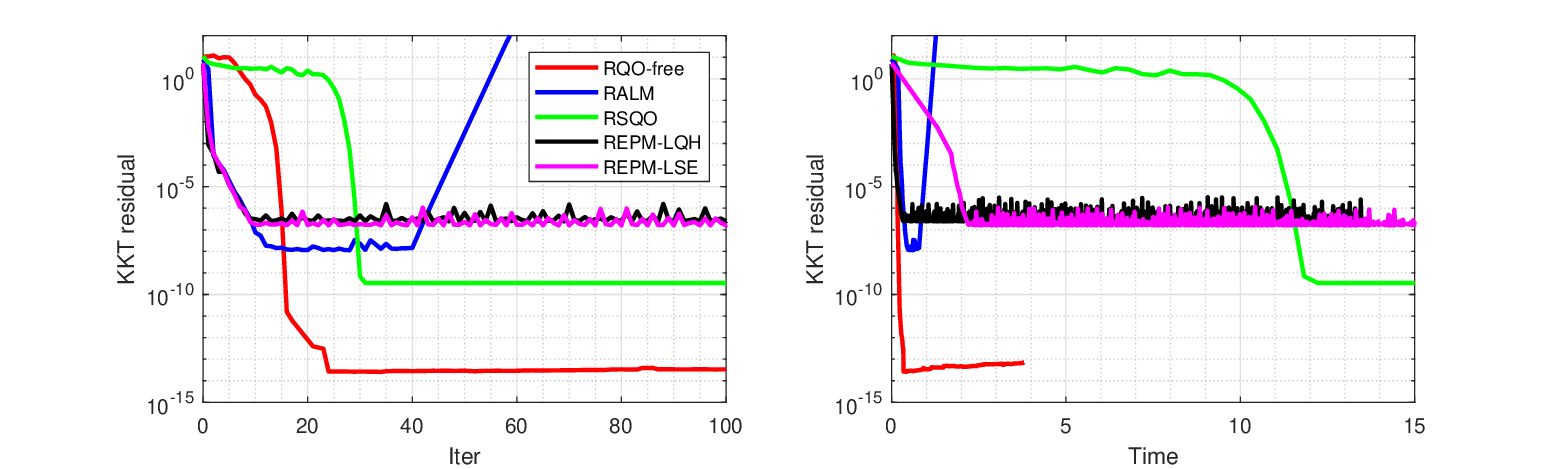}
	\caption{The convergence behavior of the KKT residuals for five algorithms on problem \eqref{low-rank} with $(d,s,r)=(10,15,4)$.}
\label{fig2}
\end{figure}


Table~\ref{tab_rank2} reports the success rates for problem \eqref{low-rank} from 20 random runs on problems with varying dimensions and average time among successful runs. Specifically, a run is called a success if the algorithm can find an iterate with a KKT residual below $10^{-7}$ before termination. The termination parameters $N_{\max}$ and $T_{\max}$ are respectively set to $2000$ and $600$ for RSQO and RQO-free and $5000$ and $600$ for RALM and REPMs.
As observed, RSQO, RALM, and RQO-free are more stable than REPMs, while RALM and RQO-free exhibit greater speed than RSQO. We further compared the performance of RSQO, RALM, and RQO-free at higher accuracy. More precisely, we redefined the success criterion as finding a solution with a KKT residual less than $5 \times 10^{-10}$ while keeping the failure criteria unchanged. Under this setting, Table \ref{tab_rank3} shows the success rate out of 20 trials, along with the average time among successful trials. It can be observed that RALM can not achieve this level of precision test, and RQO-free not only solves the problems more stably but also faster than RSQO. In summary, RQO-free is the most effective method among the compared algorithms.

{For the nonnegative PCA problem \eqref{non-pca}, we adopt the same experimental setup, except that the success criterion is redefined as achieving a KKT residual below $10^{-9}$ . Under this setting, Tables \ref{tab_pca} and \ref{tab_pca1} correspond to Tables \ref{tab_rank1} and \ref{tab_rank2}, respectively. The overall results are largely consistent with those for problem \eqref{low-rank}, except that the CPU times for RALM and the REPMs are slightly shorter.}
\begin{table}[h]
\centering
\caption{Accuracy of the compared methods on the problem \eqref{low-rank} versus number of iterations. A dash ``$-$'' indicates that the algorithm failed to reach the corresponding KKT residual. ``Res", ``Time'', ``Iter'', and ``TIter'' denote, respectively, the KKT residual, the CPU time (in seconds), the number of outer iterations, and the total number of inner iterations used for solving the penalty subproblems (for RALM and REPMs).}
\label{tab_rank1}
\setlength{\tabcolsep}{4pt}     
{\footnotesize
	\begin{tabular}{l|cc|cc|cc|cc|cc}
		\hline  
		$(d,s,r)$&\multicolumn{10}{c}{$(10,15,3)$}\\
		\hline  
		Methods&\multicolumn{2}{c|}{RQO-free}&\multicolumn{2}{c|}{RSQO}&\multicolumn{2}{c|}{RALM}&\multicolumn{2}{c|}{REPM-LQH}&\multicolumn{2}{c}{REPM-LSE}\\
		\hline
		\quad Res&Iter&Time&Iter&Time&Iter(TIter)&Time&Iter(TIter)&Time&Iter(TIter)&Time\\
		\hline
		$10^{1}$  & 1  & 0.012 & 1  & 0.427 & 1(201) & 0.495 & 1(36) & 0.101 & 1(201) & 1.163 \\
		$1$       & 1  & 0.012 & 1  & 0.427 & 1(201) & 0.495 & 1(36) & 0.101 & 1(201) & 1.163 \\
		$10^{-2}$ & 14 & 0.136 & 28 & 9.399 & 2(246) & 0.612 & 1(36) & 0.101 & 1(201) & 1.163 \\
		$10^{-4}$ & 19 & 0.165 & 29 & 9.723 & 3(271) & 0.675 & 1(36) & 0.101 & 3(434) & 2.597 \\
		$10^{-6}$ & 20 & 0.171 & 30 & 10.049 & 6(284) & 0.708 & 6(83) & 0.204 & 6(481) & 2.849 \\
		$10^{-8}$ & 21 & 0.198 & 32 & 10.688 & 10(319) & 0.783 & -- & -- & -- & -- \\
		$10^{-10}$& 21 & 0.198 & -- & -- & -- & -- & -- & -- & -- & -- \\
		$10^{-12}$& 23 & 0.214 & -- & -- & -- & -- & -- & -- & -- & -- \\
		$10^{-14}$& 30 & 0.266 & -- & -- & -- & -- & -- & -- & -- & -- \\
		$10^{-15}$& 34 & 0.289 & -- & -- & -- & -- & -- & -- & -- & -- \\
		\hline
		\hline  
		$(d,s,r)$&\multicolumn{10}{c}{$(20,30,6)$}\\
		\hline  
		Methods&\multicolumn{2}{c|}{RQO-free}&\multicolumn{2}{c|}{RSQO}&\multicolumn{2}{c|}{RALM}&\multicolumn{2}{c|}{REPM-LQH}&\multicolumn{2}{c}{REPM-LSE}\\
		\hline
		\quad Res&Iter&Time&Iter&Time&Iter(TIter)&Time&Iter(TIter)&Time&Iter(TIter)&Time\\
		\hline
		$10^{1}$   & 1  & 0.024 & 1  & 5.091 & 1(151)  & 1.029 & 1(74)  & 0.413 & 1(201)  & 3.056 \\
		$1$        & 8 & 0.351 & 2 & 10.552 & 1(151)  & 1.029 & 1(74)  & 0.413 & 1(201)  & 3.056 \\
		$10^{-2}$  & 16 & 0.641 & 36 & 184.340 & 2(194)  & 1.273 & 1(74)  & 0.413 & 1(201)  & 3.056 \\
		$10^{-4}$  & 19 & 0.751 & 37 & 189.320 & 2(194)  & 1.273 & 1(74) & 0.413 & 3(451)  & 7.196\\
		$10^{-6}$  & 20 & 0.790 & 38 & 194.160 & 6(225)  & 1.458 & 6(188) & 1.015 & 6(527)  & 8.431 \\
		$10^{-8}$  & 21 & 0.851 & 38 & 194.160 & 11(301) & 1.881 & -- & -- & -- & -- \\
		$10^{-10}$ & 21 & 0.851 & -- & -- & -- & -- & -- & -- & -- & -- \\
		$10^{-12}$ & 22 & 0.892 & -- & -- & -- & -- & -- & -- & -- & -- \\
		$10^{-14}$ & 29 & 1.147 & -- & -- & -- & -- & -- & -- & -- & -- \\
		\hline
	\end{tabular}
}
\end{table}

\begin{table}[h]
\centering
\caption{Performance of various the compared methods on problem \eqref{low-rank}.}\label{tab_rank2}
\setlength{\tabcolsep}{7pt}    
{\footnotesize
	\begin{tabular}{c|cc|cc|cc|cc|cc}
		\hline  
		Methods&\multicolumn{2}{c|}{RQO-free}&\multicolumn{2}{c|}{RSQO}&\multicolumn{2}{c|}{RALM}&\multicolumn{2}{c|}{REPM-LQH}&\multicolumn{2}{c}{REPM-LSE}\\
		\hline
		$(d,s,r)$&Rate&Time&Rate&Time&Rate&Time&Rate&Time&Rate&Time\\
		\hline
		(10,10,3)&1.00&0.405&1.00&6.825&1.00&0.722&0.15&0.409&0.25&2.251\\
		\hline
		(10,10,4)&1.00&0.067&1.00&5.998&1.00& 0.159&0.35&0.186&0.35&1.340\\   
		\hline 
		(10,20,3)&0.90&1.936&1.00& 23.041&1.00&1.353&0.20&0.924&0.10&4.496\\
		\hline
		(10,20,4)&1.00&0.190&1.00&23.495&1.00&0.659&0.10&0.383&0.05&3.470\\
		\hline
		
		(20,20,6)&1.00&0.742&1.00& 88.334&0.95& 0.678&0.10&0.722&0.10&5.797\\
		\hline    
		(20,20,8)&1.00&0.750&1.00&127.287&0.95& 0.439&0.10&0.517&0.10&6.034\\
		
		\hline
		(20,40,6)&0.90&9.032&1.00&346.847&1.00&2.979&0.10&2.779&0.20&16.812\\
		\hline
		(20,40,8)&1.00&2.886&0.70&523.267&1.00& 2.281&0.00&$-$& 0.15&21.300\\
		\hline
	\end{tabular}
}
\end{table}

\begin{table}[h]
\centering
\caption{Performance of RQO-free, RSQO and RALM on the problem \eqref{low-rank}.}\label{tab_rank3}
\begin{tabular}{l|cc|cc|cc|cc}
	\hline  
	$(d,s,r)$&\multicolumn{2}{c|}{(10,20,3)}&\multicolumn{2}{c|}{(10,20,4)}&\multicolumn{2}{c|}{(20,40,6)} &\multicolumn{2}{c}{(20,40,8)}\\
	\hline  
	Method&Rate&Time&Rate&Time&Rate&Time&Rate&Time\\
	\hline  
	RQO-free&0.85& 1.463& 0.80&0.259&0.65&8.066&0.60&3.503\\
	RSQO&0.70&323.532&0.65&325.660&0.00&$-$&0.05&418.823\\
	RALM&0.00&$-$&0.00&$-$&0.00&$-$&0.00&$-$\\
	\hline 
\end{tabular}
\end{table}



\begin{table}[h]
\centering
\caption{Accuracy of the compared methods on  problem \eqref{non-pca} versus number of iterations. }\label{tab_pca}
\setlength{\tabcolsep}{4pt}     
{\footnotesize
	\begin{tabular}{c|cc|cc|cc|cc|cc}
		\hline  
		$(d,s)$&\multicolumn{10}{c}{$(50,15)$}\\
		\hline  
		Methods&\multicolumn{2}{c|}{RQO-free}&\multicolumn{2}{c|}{RSQO}&\multicolumn{2}{c|}{RALM}&\multicolumn{2}{c|}{REPM-LQH}&\multicolumn{2}{c}{REPM-LSE}\\
		\hline
		\quad Res&Iter&Time&Iter&Time&Iter(TIter)&Time&Iter(TIter)&Time&Iter(TIter)&Time\\
		\hline
		$10^{1}$   & 1 & 0.014 & 1 & 6.792 & 1(3)   & 0.071 & 1(3)   & 0.027 & 1(3)    & 0.059 \\
		$1$        & 1 & 0.014 & 1 & 6.792 & 1(3)   & 0.071 & 1(3)   & 0.027 & 1(3)    & 0.059 \\
		$10^{-2}$  & 2 & 0.623 & 1 & 6.792 & 3(10)  & 0.143 & 1(3)   & 0.027 & 4(15)   & 0.216 \\
		$10^{-4}$  & 3 & 0.883 & 2 & 13.488 & 3(10)  & 0.143 & 1(3)   & 0.027 & 6(25)   & 0.358 \\
		$10^{-6}$  & 3 & 0.883 & 2 & 13.488 & 6(14)  & 0.185 & 4(10)  & 0.090 & 8(31)   & 0.442 \\
		$10^{-8}$  & 3 & 0.883 & 2 & 13.488 & 8(21)  & 0.255 & 4(10)  & 0.090 & 9(34)  & 0.482 \\
		$10^{-10}$ & 6 & 1.580 & 2 & 13.488 & 14(29) & 0.321 & 12(20) & 0.156 & 9(34)  & 0.482 \\
		$10^{-12}$ & 6 & 1.580 & -- & --    & 19(35) & 0.377 & 12(20) & 0.156 & 19(56)  & 0.753 \\
		$10^{-13}$ & 6 & 1.580 & -- & --    & -- & --     & -- & --     & -- & --      \\
		\hline
		\hline  
		\quad$(d,s)$&\multicolumn{10}{c}{(70,20)}\\
		\hline  
		Methods&\multicolumn{2}{c|}{RQO-free}&\multicolumn{2}{c|}{RSQO}&\multicolumn{2}{c|}{RALM}&\multicolumn{2}{c|}{REPM-LQH}&\multicolumn{2}{c}{REPM-LSE}\\
		\hline
		\quad Res&Iter&Time&Iter&Time&Iter(TIter)&Time&Iter(TIter)&Time&Iter(TIter)&Time\\
		\hline
		$10^{1}$   & 1 & 0.044 & 1 & 32.183 & 1(3)  & 0.205 & 1(3)  & 0.069 & 1(3)   & 0.150 \\
		$1$        & 1 & 0.044 & 1 & 32.183 & 1(3)  & 0.205 & 1(3)  & 0.069 & 1(3)   & 0.150 \\
		$10^{-2}$  & 2 & 1.749 & 1 & 32.183 & 3(9)  & 0.375 & 1(3)  & 0.069 & 5(16)  & 0.573 \\
		$10^{-4}$  & 3 & 3.458 & 2 & 64.881 & 3(9)  & 0.375 & 1(3)  & 0.069 & 7(23)  & 0.787 \\
		$10^{-6}$  & 3 & 3.458 & 2 & 64.881 & 3(9)  & 0.375 & 5(11) & 0.251 & 8(26)  & 0.887 \\
		$10^{-8}$  & 3 & 3.458 & 2 & 64.881 & 6(14) & 0.489 & 5(11) & 0.251 & 9(29)  & 0.982 \\
		$10^{-10}$ & 3 & 3.458 & -- & --    & 6(14) & 0.489     & -- & --     & -- & --      \\
		$10^{-12}$ & 4 & 4.801 & -- & --     & -- & --     & -- & --     & -- & --      \\
		$10^{-13}$ & 4 & 4.801 & -- & --     & -- & --     & -- & --     & -- & --      \\
		\hline
	\end{tabular}
}
\end{table}

\begin{table}[h]
	\centering
	\caption{Performance of various the compared methods on problem \eqref{non-pca}.}\label{tab_pca1}
    \setlength{\tabcolsep}{7pt}  
    {\tiny
	\begin{tabular}{c|cc|cc|cc|cc|cc}
		\hline  
		Methods&\multicolumn{2}{c|}{RQO-free}&\multicolumn{2}{c|}{RSQO}&\multicolumn{2}{c|}{RALM}&\multicolumn{2}{c|}{REPM-LQH}&\multicolumn{2}{c}{REPM-LSE}\\
		\hline
        $(d,s)$&Rate&Time&Rate&Time&Rate&Time&Rate&Time&Rate&Time\\
\hline
$(50,20)$ & 1.00 & 2.040 & 1.00 & 17.961 & 0.75 & 0.276
            & 0.80 & 0.124 & 0.65 & 0.462 \\
\hline            
$(70,20)$ & 1.00 & 4.450 & 0.25 & 40.600 & 0.55 & 0.533
            & 0.70 & 0.319 & 0.50 & 0.915 \\
            \hline
$(90,20)$ & 1.00 & 4.676 & 0.15 & 74.623 & 0.75 & 0.665
            & 0.95 & 0.319 & 0.70 & 1.196 \\
            \hline
$(70,10)$ & 1.00 & 3.353 & 1.00 & 8.884 & 0.70 & 0.182
            & 0.45 & 0.107 & 0.55 & 0.385 \\
            \hline
$(70,30)$ & 1.00 & 2.636 & 0.00 & -- & 0.65 & 0.482
            & 0.65 & 0.404 & 0.50 & 1.600 \\
  \hline          
	\end{tabular}
    }
\end{table}

\section{Conclusions} \label{Sec:Conclusion}

In this paper, we have presented a QO-free method for solving Riemannian optimization problems with equality and inequality constraints.
At each iteration, the master search direction is obtained by 
solving three linear systems, all of which are always consistent.
This is not computationally expensive, as these systems share the same linear operator.
In order to avoid the Maratos effect, the correction direction is generated by solving a least squares problem, which can be neglected in early iterations.
Global convergence as well as strong and superlinear convergence have been proved under mild assumptions. Promising preliminary numerical results were reported.

As future work, the proposed method can be further improved. For instance, we may (i) reduce the scale of the linear systems by introducing the active set strategy (see, e.g., \cite{chen2006feasible}), (ii) extend it to handle infeasible initial points by the idea of the strongly subfeasible direction methods (see, e.g., \cite{Jian2010sequential,Tang2012}), and (iii) incorporate a trial line search to check whether $t_k=1$ is acceptable, 
which can reduce the number of times the least squares problem is solved and the arc search is executed (see, e.g., \cite{jian2010new}).

In addition, it is of interest to apply our method to other practical problems, such as multicontact posture computation for static robots \cite{brossette2018multicontact} and stable linear system identification \cite{obara2024stable}; however, this lies beyond the scope of this paper.

\vspace{0.3cm}

\noindent\textbf{Funding}

\noindent 
This work was supported by the Guangxi Natural Science Foundation (2026GXNSFDA00640024), 
the National Natural Science Foundation of China (Nos. 12271113, 12371311), the Natural Science Foundation of Fujian Province (No. 2023J06004), the Fundamental Research Funds for the Central Universities (No. 20720240151), and Xiaomi Young Talents Program.

\bigskip
\bibliography{sn-bibliography.bib}

\appendix

\noindent
\section*{Appendix}  
\section{The Proofs in Section \ref{Sec:algorithm}}\label{AppendixA}
\noindent
{\bf{Proof of Proposition \ref{tp-p} }}
(i) Since $(x^*,\lambda^*)$ is a KKT pair for \eqref{TRpro}, we obtain
\begin{subequations}\label{3kkt}
\begin{align}
	\label{3kkt1}
	&{\rm grad}F_\rho(x^*)+\sum_{i\in \mathcal{L}}\lambda^*_i{\rm grad}c_i(x^*)=0_{x^*},\\
	\label{3kkt2}
	&\lambda^*_i\geq 0,\,\,c_i(x^*)\leq 0,\quad \forall i\in \mathcal{L},\\
	\label{3kkt3}
	&\lambda^*_ic_i(x^*)=0,\quad \forall i\in \mathcal{L}.
\end{align}	
\end{subequations}
It follows from \eqref{3kkt1} that
$${\rm grad}f(x^*)+\sum_{i\in \mathcal{I}}\lambda^*_i{\rm grad}c_i(x^*)+\sum_{i\in\mathcal{E}}(\lambda^*_i-\rho){\rm grad}c_i(x^*)=0_{x^*}.$$
This, together with the fact that $c_i(x^*)=0$ for $i\in\mathcal{E}$, as well as \eqref{3kkt2} and \eqref{3kkt3}, establishes the conclusion.

(ii) Note that $(x^*,\lambda_\mathcal{I}^*,\lambda_\mathcal{E}^*-\rho e)$ is a KKT pair of problem \eqref{Rpro}, thus
\[
\begin{aligned}
&{\rm grad}F_\rho(x^*)+\sum_{i\in \mathcal{L}}\lambda^*_i{\rm grad}c_i(x^*)\\
=&{\rm grad}f(x^*)+\sum_{i\in \mathcal{I}}\lambda^*_i{\rm grad}c_i(x^*)+\sum_{i\in\mathcal{E}}(\lambda^*_i-\rho){\rm grad}c_i(x^*)\\
=&0_{x^*},
\end{aligned}
\]
and $\lambda_{\mathcal{I}}^*\ge 0$, $c_i(x^*)\le 0$, $\lambda^*_ic_i(x^*)=0$ for all $i\in\mathcal{I}$. Therefore, the conclusion is follows since $\lambda_{\mathcal{E}}^*\ge 0$ and $c_i(x^*)=0$ for all $i\in\mathcal{E}$.\qed

\vspace{0.5cm}
\noindent
{\bf{Proof of Lemma \ref{le1} }} 
Fix an arbitrary index $k$, and let $(\bar{\eta},\bar{\lambda})$ be a solution to the system 
$$\mathcal{A}_k(\eta,\lambda)=(0_{x^k},0).$$
This implies
\begin{eqnarray}\label{proof10}
&&\mathcal{H}_k[\bar{\eta}]+\sum_{i\in \mathcal{L}}\bar{\lambda}_i{\rm grad}c_i(x^k)=0_{x^k},\nonumber\\
&&\alpha_i^k\langle{\rm grad}c_i(x^k),\bar{\eta}\rangle-\sqrt{2}\beta_i^k\bar{\lambda}_i=0,\quad \forall i\in \mathcal{L}.
\end{eqnarray}
From \eqref{alha-beta}, we thus have
$$\left\langle \bar{\eta},\mathcal{H}_k[\bar{\eta}]\right\rangle+\sum_{i\in  \mathcal{L}}\dfrac{\alpha_i^k\langle{\rm grad}c_i(x^k),\bar{\eta}\rangle^2}{\sqrt{2}\beta_i^k}=0.$$
Since $\alpha_i^k\geq 0$ for all $i\in \mathcal{L}$, we obtain $\left\langle \bar{\eta},\mathcal{H}_k[\bar{\eta}]\right\rangle=0$. 
This together with the positive definiteness of $\mathcal{H}_k$ shows that $\bar{\eta}=0_{x^k}$.
In addition, it follows from \eqref{proof10} that $\bar{\lambda}=0$. Thus, the proof is complete.\qed

\vspace{0.5cm}
\noindent
{\bf{Proof of Lemma \ref{le3}}}
If $\eta^{k1}$ and $\eta^k$ have been generated by Algorithm~\ref{algo1}, then we can conclude that $\eta^{k0}\neq 0_{x^k}$. In fact, if $\eta^{k0}= 0_{x^k}$, then it follows from \eqref{s12} that $\sqrt{2}\beta_i^k\lambda_i^{k0}=0$ for all $i\in \mathcal{L}$. This along with \eqref{alha-beta} shows that $\lambda^{k0}=0$. 
Thus, the conditions (a)-(c) in the step 3 of Algorithm \ref{algo1} are all satisfied, and the algorithm cannot proceed to produce $\eta^{k1}$ and $\eta^k$. This is a contradiction.

(i) By \eqref{s11} and \eqref{s21}, we obtain
\begin{eqnarray}
\label{proof30}
&&\langle{\rm grad}F_{\rho_k}(x^k),\eta^{k1}\rangle
=  -\langle \eta^{k1},\mathcal{H}_k[\eta^{k0}]\rangle-\sum_{i\in \mathcal{L}}\lambda_i^{k0}\langle{\rm grad}c_i(x^k),\eta^{k1}\rangle,\\
\label{proof31}
&&\langle{\rm grad}F_{\rho_k}(x^k),\eta^{k0}\rangle
=  -\langle \eta^{k1},\mathcal{H}_k[\eta^{k0}]\rangle-\sum_{i\in \mathcal{L}}\lambda_i^{k1}\langle{\rm grad}c_i(x^k),\eta^{k0}\rangle.
\end{eqnarray}
For $i\in \mathcal{L}$, if $\alpha_i^k=0$, then it follows from \eqref{s12} and \eqref{alha-beta} that  $\lambda_i^{k0}=0$. This in turn implies $\lambda_i^{k1}=0$ by \eqref{s22}. If $\alpha_i^k\neq0$, from \eqref{s12} and \eqref{s22} we have
\begin{eqnarray*}
&&\langle{\rm grad}c_i(x^k),\eta^{k0}\rangle=\dfrac{\sqrt{2}\beta_i^k\lambda_i^{k0}}{\alpha_i^k},\\
&&\langle{\rm grad}c_i(x^k),\eta^{k1}\rangle=\dfrac{\sqrt{2}\beta_i^k\lambda_i^{k1}}{\alpha_i^k}+(\min\{\lambda_i^{k0},0\})^3.
\end{eqnarray*}
This along with \eqref{proof30} and \eqref{proof31} shows that
\begin{equation*}
\begin{aligned}
	\langle{\rm grad}F_{\rho_k}(x^k),\eta^{k1}\rangle&=\langle{\rm grad}F_{\rho_k}(x^k),\eta^{k0}\rangle-\mathop\sum\limits_{i:\lambda_i^{k0}\neq 0}\lambda_i^{k0}(\min\{\lambda_i^{k0},0\})^3\\
	&=\langle{\rm grad}F_{\rho_k}(x^k),\eta^{k0}\rangle-\mathop\sum\limits_{i:\lambda_i^{k0}<0}(\lambda_i^{k0})^4 < 0.
\end{aligned}
\end{equation*}
where the last inequality is due to \eqref{dk01}.

(ii) The definitions of $\eta^k$ and $\theta_k$ show that
\begin{equation*}
\begin{aligned}
	\langle{\rm grad}F_{\rho_k}(x^k),\eta^{k}\rangle&=\langle{\rm grad}F_{\rho_k}(x^k),\eta^{k1}\rangle-\theta_k\langle{\rm grad}F_{\rho_k}(x^k),\eta^{k1}-\eta^{k2}\rangle\\
	&\leq\tau\langle{\rm grad}F_{\rho_k}(x^k),\eta^{k1}\rangle <0,
\end{aligned}
\end{equation*}
which completes this proof.\qed

\vspace{0.5cm}
\noindent
{\bf{Proof of Lemma \ref{le4} } }
We divide the proof into two parts.

(i) For each $k$, using the Taylor expansion (see \cite[Sec. 4.1]{Boumal2023}) of $F_{\rho_k}$ around $x^k$, we have
\begin{eqnarray*}
&&F_{\rho_k}(R_{x^k}(t\eta^k+t^2\tilde{\eta}^k))-F_{\rho_k}(x^k)-\sigma t\langle{\rm grad}F_{\rho_k}(x^k),\eta^k\rangle\\
&=&F_{\rho_k}(x^k)+t\langle{\rm grad}F_{\rho_k}(x^k),\eta^k\rangle-F_{\rho_k}(x^k)-\sigma t\langle{\rm grad}F_{\rho_k}(x^k),\eta^k\rangle+o(t)\\
&=&(1-\sigma)t\langle{\rm grad}F_{\rho_k}(x^k),\eta^k\rangle+o(t).
\end{eqnarray*}
This, together with the facts that $\langle{\rm grad}F_{\rho_k}(x^k),\eta^{k}\rangle<0$ (by Lemma \ref{le3}) and $\sigma\in(0,1/2)$, shows that there exists a $\bar{t}_0>0$ such that \eqref{linef} holds for all $t\in[0,\bar{t}_0]$.

(ii) For each $i\in\mathcal{L}$, since $x^k\in{\rm int}\mathcal{F}$ and $c_i$ is continuous, there exists a neighborhood $U_i^k$ of $x^k$ such that $c_i(x)<0$ for all $x\in U_i^k$. By the continuity of $R$, there exists a $\bar{t}_i>$ such that $R_{x^k}(t\eta^k+t^2\tilde{\eta}^k)\in U_i^k$ for all $t\in[0,\bar{t}_i]$; that is, $c_i(R_{x^k}(t\eta^k+t^2\tilde{\eta}^k))< 0$ for all $t\in[0,\bar{t}_i]$. 

In summary, let $\bar{t}=\min\{\bar{t}_0,\bar{t}_1,\cdots,\bar{t}_{m+\ell}\}$.
Then, the arc search conditions \eqref{linef} and \eqref{lineg}  are both satisfied for all sufficiently small $t$ with $t\leq \bar{t}$.\qed

\section{The Proofs in Section \ref{Sec:global}}\label{AppendixB}
{\bf{Proof of Lemma \ref{facanshu}}} Suppose by contradiction that $\rho_k$ is updated infinitely many times. Then, there exists an infinite index set $\mathcal{K}\subset \mathbb{N}$ such that for all $k\in \mathcal{K}$, $\rho_{k+1}>\rho_k$; this implies
\begin{eqnarray}
\label{shi1}
&&\|\eta^{k0}\|\leq r_1,\\
\label{shi2}
&&\lambda^{k0}_\mathcal{E}\ngeq r_2e,\\
\label{shi3}
&&\lambda^{k0}_\mathcal{L}\geq -r_3e,
\end{eqnarray}
for all $k\in \mathcal{K}$. On the other hand, since $\eta^{k0}$ and $\lambda^{k0}$ are solutions to \eqref{s1}, we have
\begin{equation}\label{shi5}
\mathcal{H}_k[\eta^{k0}]+{\rm grad}f(x^k)+\sum_{i\in \mathcal{E}}(\lambda_i^{k0}-\rho_k){\rm grad}c_i(x^k)+\sum_{i\in \mathcal{I}}\lambda_i^{k0}{\rm grad}c_i(x^k)=0_{x^k},
\end{equation}
\begin{equation}\label{shi6}
\alpha_i^k\langle{\rm grad}c_i(x^k),\eta^{k0}\rangle-\sqrt{2}\beta_i^k\lambda_i^{k0}=0,\quad \forall\, i\in \mathcal{L}.
\end{equation}
It follows from the update way of $\rho_k$ that $\rho_k\rightarrow+\infty$ as $k\rightarrow\infty$. Then by \eqref{shi2}, we have $\{\|\lambda_{\mathcal{E}}^{k0}-\rho_ke\|_\infty\}_{\mathcal{K}}\rightarrow+\infty$, where $\|\cdot\|_\infty$ denotes the infinity norm. Now, we define
$$z_k=\max\{\|\lambda_{\mathcal{E}}^{k0}-\rho_ke\|_\infty,\|\lambda_{\mathcal{I}}^{k0}\|_\infty,1\},$$
and then it follows
\begin{equation}\label{shi6.5}
\{z_k\}_{\mathcal{K}}\rightarrow+\infty.
\end{equation}
For any $k\in \mathcal{K}$, we denote
\begin{equation}\label{shi7}
\hat{\lambda}_{\mathcal{E}}^{k}=\dfrac{\lambda_{\mathcal{E}}^{k0}-\rho_ke}{z_k}
\quad {\rm and } \quad
\hat{\lambda}_{\mathcal{I}}^{k}=\dfrac{\lambda_{\mathcal{I}}^{k0}}{z_k}.
\end{equation}		
By \eqref{shi6.5} and the definition of $z_k$, we know that $\max\{\|\hat{\lambda}_{\mathcal{E}}^{k}\|_\infty,\|\hat{\lambda}_{\mathcal{I}}^{k}\|_\infty\}=1$ for all $k\in \mathcal{K}$ large enough. It follows from Assumption \ref{A2} that $\{x^k\}_{\mathcal{K}}$ is bounded. Thus, there exist an infinite index set $\mathcal{K}'\subset\mathcal{K}$ and $x^*\in\mathcal{F}$, $\hat{\lambda}_{\mathcal{E}}^*\in\mathbb{R}^\ell$, $\hat{\lambda}_{\mathcal{I}}^*\in\mathbb{R}^m$, where $\hat{\lambda}_{\mathcal{E}}^*$ and $\hat{\lambda}_{\mathcal{I}}^*$ are not simultaneously the zero vectors, such that
$$\{x^k\}_{\mathcal{K}'}\rightarrow x^*,\quad \{\hat{\lambda}_{\mathcal{E}}^{k}\}_{\mathcal{K}'}\rightarrow \hat{\lambda}_{\mathcal{E}}^*,\quad \{\hat{\lambda}_{\mathcal{I}}^{k}\}_{\mathcal{K}'}\rightarrow \hat{\lambda}_{\mathcal{I}}^*.$$

On the other hand, for all $i\in \mathcal{L}$, by the boundedness of $\{x^k\}_{\mathcal{K}}$ and the twice continuous differentiability of $c_i(x)$, the sequence $\{\|{\rm grad}c_i(x^k)\|\}_{\mathcal{K}}$ is bounded. 
Since $\{\mu_i^k\}_{\mathcal{K}'}$ and $\{\beta_i^k\}_{\mathcal{K}'}$ are bounded, without loss of generality, we can assume that $\{\mu_i^k\}_{\mathcal{K}'}\rightarrow \mu_i^*$ and $\{\beta_i^k\}_{\mathcal{K}'}\rightarrow \beta_i^*$. 
It then follows that 
$$\beta_i^*=\left(1-\dfrac{\mu_i^*}{\sqrt{c_i^2(x^*)+(\mu_i^*)^2}}\right)^{1/2}>0,\quad \forall\, i\notin \mathcal{L}(x^*).$$

For $i\in \mathcal{L}$, dividing both sides of \eqref{shi6} by $z_k$ and taking the limit over $\mathcal{K}'$, from the boundedness of $\{\alpha_i^k\}_{\mathcal{K}'}$ and \eqref{shi1}, we have
$$\left\{\dfrac{\beta_i^k\lambda_i^{k0}}{z_k}\right\}_{\mathcal{K}'}\rightarrow 0.$$
For $i\notin\mathcal{E}(x^*)$, we have $\beta_i^*>0$, and thus
$$\left\{\dfrac{\lambda_i^{k0}}{z_k}\right\}_{\mathcal{K}'}\rightarrow 0,\ \forall i\notin\mathcal{E}(x^*).$$
It follows from \eqref{shi7} that there exists a constant $w\geq 0$ such that $\{\rho_{k}/z_k\}_{\mathcal{K}'}\rightarrow w$ and 
\begin{equation}\label{shi8.5}
\hat{\lambda}_i^*=-w, \ \forall i\notin\mathcal{E}(x^*).
\end{equation}
Combining \eqref{shi3} and \eqref{shi7}, we obtain
\begin{equation}\label{shi9}
\hat{\lambda}_i^*\geq-w, \quad\forall\, i\in\mathcal{E};\quad \hat{\lambda}_\mathcal{I}^*\geq 0.
\end{equation}		
Similarly, for $i\notin\mathcal{I}(x^*)$, we obtain
$$\left\{\dfrac{\lambda_i^{k0}}{z_k}\right\}_{\mathcal{K}'}\rightarrow  \hat{\lambda}_{i}^*=0,\quad i\notin\mathcal{I}(x^*).$$

Finally, dividing both sides of \eqref{shi5} by $z_k$, taking the limit over $\mathcal{K}'$,  and combining \eqref{shi1}, Assumption \ref{A3}, and the twice continuous differentiability of $f(x)$, we obtain
\begin{equation}\label{shi10}
\sum_{i\in \mathcal{E}}\hat{\lambda}_{i}^*{\rm grad}c_i(x^*)+\sum_{i\in \mathcal{I}(x^*)}\hat{\lambda}_{i}^*{\rm grad}c_i(x^*)=0_{x^*}.
\end{equation}
Since $\hat{\lambda}_{\mathcal{E}}^*$ and $\hat{\lambda}_{\mathcal{I}}^*$ are not both zero, it follows from Assumption \ref{A1} that $\mathcal{E}(x^*)\neq\mathcal{E}$ (i.e., $x^*\in\mathcal{F}\backslash\mathcal{F}_P$) and $w>0$. Dividing both sides of \eqref{shi10} by $-w$ and combining with \eqref{shi8.5}, we have 
\begin{equation*}
\sum_{i\in \mathcal{E}}{\rm grad}c_i(x^*)-\sum_{i\in \mathcal{E}(x^*)}(\dfrac{\hat{\lambda}_{i}^*}{w}+1){\rm grad}c_i(x^*)-\sum_{i\in \mathcal{I}(x^*)}\dfrac{\hat{\lambda}_{i}^*}{w}{\rm grad}c_i(x^*)=0_{x^*}.
\end{equation*}
This together with \eqref{shi9} contradicts Assumption \ref{A3.5}; thus, the proof is completed.\qed

\vspace{0.5cm}
\noindent
{\bf{Proof of Lemma \ref{le6}}} (i) From Assumption \ref{A2}, we see that the sequence $\{{\rm grad}F_{\bar{\rho}}(x^k)\}$ is bounded. This together with Lemma \ref{le5} and  \eqref{relationship1} shows that $\{(\eta^{k0},\lambda^{k0})\}$ is bounded. 
It then follows from the boundedness of $\{\alpha_i^k\},i\in \mathcal{L}$ shows that $\{v^{k1}\}$ is bounded. So $\{(\eta^{k1},\lambda^{k1})\}$ is bounded by \eqref{relationship2}. Furthermore, the boundedness of $\{\eta^{k1}\}$ shows that $\{v^{k2}\}$ is bounded. Therefore $\{(\eta^{k2},\lambda^{k2})\}$ is bounded by \eqref{relationship3}.

(ii) Since $\langle{\rm grad}F_{\bar{\rho}}(x^k),\eta^{k1}\rangle<0$ (see Lemma \ref{le3}), we have $\theta_k\leq 1$ for all $k$ by its definition. 
Let $C_1={\rm sup}\|\mathcal{A}_k^{-1}\|_{\rm op}$. Then we have $C_1<+\infty$ by Lemma \ref{le5}. Define $$\Delta\eta^k=\eta^k-\eta^{k1}\,\,{\rm and\,\,} \Delta\lambda^k=\lambda^k-\lambda^{k1}.$$
Note that the operator $\mathcal{A}_k^{-1}$ is a linear transformation. Then by the definition of $\eta^k$ and \eqref{relationship}, we obtain
\begin{equation*}
(\Delta\eta^k,\Delta\lambda^k)=\theta_k\mathcal{A}_k^{-1}(0_{x^k},\Delta v^{k}),
\end{equation*}
where $\Delta v^{k}=v^{k2}-v^{k1}=(-\alpha_1^k\|\eta^{k1}\|^\nu,\cdots,-\alpha_{m+\ell}^k\|\eta^{k1}\|^\nu)^{\top}$. 
Therefore, from \eqref{alha-beta}, it follows
$$\|(\Delta\eta^k,\Delta\lambda^k)\|\leq \sqrt{m+\ell}C_1\|\eta^{k1}\|^\nu.$$
Denote $C=\sqrt{m+\ell}C_1$. Then we have $\|\eta^k-\eta^{k1}\|=\|\Delta\eta^k\|\leq\|(\Delta\eta^k,\Delta\lambda^k)\|\leq C\|\eta^{k1}\|^\nu$.\qed

\vspace{0.5cm}
\noindent
{\bf{Proof of Lemma \ref{le7}}} From Lemma \ref{le3},  we have
$$\langle{\rm grad}F_{\bar{\rho}}(x^k),\eta^{k}\rangle\leq \tau\langle{\rm grad}F_{\bar{\rho}}(x^k),\eta^{k0}\rangle-\tau\mathop\sum\limits_{i:\lambda_i^{k0}<0}(\lambda_i^{k0})^4.$$
This together with \eqref{dk01} and Assumption \ref{A3} shows that
$$\langle{\rm grad}F_{\bar{\rho}}(x^k),\eta^{k}\rangle\leq -\tau a_1\|\eta^{k0}\|^2-\tau\mathop\sum\limits_{i:\lambda_i^{k0}<0}(\lambda_i^{k0})^4.$$
Therefore by $\{\eta^k\}_{\mathcal{K}}\rightarrow 0_{x^*}$, we obtain
\begin{equation}\label{proof-le7-1}
\{\eta^{k0}\}_{\mathcal{K}}\rightarrow 0_{x^*}\quad {\rm and }\quad \left\{\mathop\sum\limits_{i:\lambda_i^{k0}<0}(\lambda_i^{k0})^4\right\}_{\mathcal{K}}\rightarrow 0.
\end{equation}
By Lemma \ref{le6} (i) and \eqref{gamma}, we know that $\{\lambda^{k0}\}_{\mathcal{K}}$ and $\{\mu^k\}_{\mathcal{K}}$ are all bounded. Thus there exists $\mathcal{K}'\subseteq \mathcal{K}$ such that $\{\lambda^{k0}\}_{\mathcal{K}'}\rightarrow \lambda^*$, $\{\mu^k\}_{\mathcal{K}'}\rightarrow\mu^*$ and $\{\beta_i^k\}_{\mathcal{K}'}\rightarrow\beta_i^*,i\in \mathcal{L}$. It follows from \eqref{proof-le7-1} that $\lambda^*_i\geq 0$ for all $i\in \mathcal{L}$. 
In addition, since $\{x^k\}\subseteq{\rm int}\mathcal{F}$, we have $x^*\in\mathcal{F}$.

On the other hand, taking the limits on both sides of \eqref{s1} in $\mathcal{K}'$, by Assumption \ref{A3} and $\{\eta^{k0}\}_{\mathcal{K}}\rightarrow 0_{x^*}$, we obtain
\begin{equation}\label{proof-le7-2}
\sum_{i=1}^{m+\ell}\lambda_i^*{\rm grad}c_i(x^*)=-{\rm grad}F_{\bar{\rho}}(x^*)\quad {\rm and}\quad \beta_i^*\lambda_i^*=0,\quad i\in \mathcal{L}.
\end{equation}
For $i\notin \mathcal{L}(x^*)$, we have  
$$\beta_i^*=\left(1-\dfrac{\mu_i^*}{\sqrt{c_i^2(x^*)+(\mu_i^*)^2}}\right)^{1/2}>0.$$
This along with \eqref{proof-le7-2} shows that $\lambda_i^*=0$ for all $i\notin \mathcal{L}(x^*)$. 

Summarizing the above analysis, we can conclude that $(x^*,\lambda^*)$ is a KKT pair of problem \eqref{TRpro}. Next, we show that $\{\lambda^{k0}\}_{\mathcal{K}}\rightarrow\lambda^*$. In fact, if there is another accumulation point of $\{\lambda^{k0}\}_{\mathcal{K}}$, denoted by $\bar{\lambda}^*$. From the above analysis, we can immediately conclude that $\bar{\lambda}^*$ satisfies \eqref{proof-le7-2} upon replacing  
$\lambda^*$ with $\bar{\lambda}^*$,  
and $\bar{\lambda}^*_i=0=\lambda_i^*$ for all $i\notin \mathcal{L}(x^*)$. In addition, for $i\in \mathcal{L}(x^*)$, by Assumption \ref{A1}, we obtain $\bar{\lambda}^*_i=\lambda_i^*$. Thus $\bar{\lambda}^*=\lambda^*$.\qed

\vspace{0.5cm}
\noindent
{\bf{Proof of Lemma \ref{le8}}} Let $x^{**}$ be an accumulation point of $\{x^{k-1}\}_\mathcal{K}$. Then there exists an index set $\mathcal{K}'\subseteq \mathcal{K}$ such that $\{x^{k-1}\}_{\mathcal{K}'}\rightarrow x^{**}$. By Lemma \ref{le7}, it follows that $x^{**}$ is a KKT point of problem \eqref{TRpro}. Thus we only need to prove that $x^{**}=x^*$. Since for all $k\in\mathcal{K}'$, we have $x^k=R_{x^{k-1}}(t_{k-1}\eta^{k-1}+t_{k-1}^2\tilde{\eta}^{k-1})$. 
This together with the continuity of $R$, $\|\tilde{\eta}^{k-1}\|\leq\|\eta^{k-1}\|$ and $\{\|\eta^{k-1}\|\}_\mathcal{K}\rightarrow 0$ implies that
$$\lim_{k\in\mathcal{K}'} x^k=\lim_{k\in\mathcal{K}'}R_{x^{k-1}}(t_{k-1}\eta^{k-1}+t_{k-1}^2\tilde{\eta}^{k-1})=R_{x^{**}}(0_{x^{**}})=x^{**}.$$ Thus $x^{**}=x^*$, which completes the proof.\qed

\vspace{0.5cm}
\noindent
{\bf{Proof of Lemma \ref{le9}}} Suppose by contradiction that $x^*$ is not a KKT point of \eqref{TRpro}. Note that for all $k\in\mathcal{K}$, $\langle{\rm grad}F_{\bar{\rho}}(x^k),\eta^{k1}\rangle<0$ (see Lemma \ref{le3}). We further show that there exists a constant $\varepsilon>0$ such that 
\begin{equation}\label{proof-le9-1}
\langle{\rm grad}F_{\bar{\rho}}(x^k),\eta^{k1}\rangle\leq-\varepsilon,\,\, \forall k\in\mathcal{K}.
\end{equation}
In fact, if there is an index set $\mathcal{K}'\subseteq\mathcal{K}$ such that $$\lim_{k\in\mathcal{K}'}\left\langle{\rm grad}F_{\bar{\rho}}(x^k),\eta^{k1}\right\rangle\rightarrow 0,$$ 
then by the same lines as Lemma \ref{le7}, one can prove that $x^*$ is a KKT point of \eqref{TRpro}, which is a contradiction. Thus, from Lemma \ref{le3}, we obtain 
\begin{equation}\label{proof-le9-4}
\langle{\rm grad}F_{\bar{\rho}}(x^k),\eta^{k}\rangle\leq-\tau\varepsilon,\,\, \forall k\in\mathcal{K}.
\end{equation}
Note that $\{\eta^k\}$ and $\{\tilde{\eta}^k\}$ are both bounded (see Lemma \ref{le6} and the definition of $\tilde{\eta}^k$, respectively). Using the Taylor expansion and combining with \eqref{proof-le9-4} and $\sigma\in(0,1/2)$, we have
\begin{equation*}
\begin{aligned}
	&F_{\bar{\rho}}(R_{x^k}(t\eta^k+t^2\tilde{\eta}^k))- F_{\bar{\rho}}(x^k)-\sigma t\langle{\rm grad}F_{\bar{\rho}}(x^k),\eta^k\rangle\\
	=&F_{\bar{\rho}}(x^k)+t\langle{\rm grad}F_{\bar{\rho}}(x^k),\eta^k\rangle- F_{\bar{\rho}}(x^k)-\sigma t\langle{\rm grad}F_{\bar{\rho}}(x^k),\eta^k\rangle+o(t)\\
	\leq&-(1-\sigma)\tau\varepsilon t+o(t).
\end{aligned}
\end{equation*}
Thus there exists a  $\underline{t}_0>0$ such that the following inequality holds for all  $t\in[0,\underline{t}_0]$:
\begin{equation*}\label{proof-le9-2}
F_{\bar{\rho}}(R_{x^k}(t\eta^k+t^2\tilde{\eta}^k))\leq F_{\bar{\rho}}(x^k)+\sigma t\left\langle{\rm grad}F_{\bar{\rho}}(x^k),\eta^k\right\rangle.
\end{equation*}

On the other hand, for each $i\notin \mathcal{L}(x^*)$, there is a constant $\epsilon_i>0$ such that
$$c_i(x^k)<-\epsilon_i$$
for all $k\in\mathcal{K}$ large enough. Therefore, by the boundedness of $\{\eta^k\}$ and $\{\tilde{\eta}^k\}$, and the continuity of $c_i$ and $ R$, there exists a $\underline{t}_i>0$ such that 
$$c_i(R_{x^k}(t\eta^k+t^2\tilde{\eta}^k))<0$$
for all $t\in[0,\underline{t}_i]$ and $k\in\mathcal{K}$ large enough.

In addition, for each $i\in \mathcal{L}(x^*)$,  by Taylor expansion, we have
\begin{equation}\label{proof-le9-6}
c_i(R_{x^k}(t\eta^k+t^2\tilde{\eta}^k))=c_i(x^k)+t\langle{\rm grad}c_i(x^k),\eta^k\rangle+o(t).
\end{equation}
From \eqref{gamma} and $\inf_{k\in{\mathcal{K}}}\|\eta^{k-1}\|> 0$, it follows that there is a constant $\underline{\mu}$ such that $\mu_i^k\geq\underline{\mu}>0$ for all $k\in\mathcal{K}$.
Thus $\alpha_i^k>0$ for all $k\in\mathcal{K}$. By \eqref{s42}, we have
\begin{equation}\label{proof-le9-3}
\langle{\rm grad}c_i(x^k),\eta^k\rangle
=(\min\{\lambda_i^{k0},0\})^3+\dfrac{\sqrt{2}\beta_i^k\lambda_i^k}{\alpha_i^k}-\theta_k\|\eta^{k1}\|^\nu.
\end{equation}
Moreover, for $i\in \mathcal{L}(x^*)$, we obtain
$$\{c_i(x^k)\}_{\mathcal{K}}\rightarrow 0, \quad \{\alpha_i^k\}_{\mathcal{K}}\rightarrow 1,$$
and
\begin{equation}\label{proof-le9-5}
\delta_i^k=\left(\dfrac{1}{\sqrt{c_i^2(x^k)+(\mu_i^k)^2}\left(\sqrt{c_i^2(x^k)+(\mu_i^k)^2}+\mu_i^k\right)}\right)^{1/2}\leq\dfrac{1}{\sqrt{2}\underline{\mu}}
\end{equation}
for all $k\in\mathcal{K}$. 
By the boundedness of $\{\eta^{k1}\}$ and $\{\eta^{k2}\}$, and the continuity of ${\rm grad}F_{\bar{\rho}}$, it follows that $\{\langle{\rm grad}F_{\bar{\rho}}(x^k),\eta^{k1}-\eta^{k2}\rangle\}$
is bounded. This together with \eqref{proof-le9-1} and the definition of $\theta_k$ shows that there is a constant $\bar{\theta}\in(0,1)$ such that $\theta_k>\bar{\theta}$ for all $k\in\mathcal{K}$ large enough. 
It follows from \eqref{proof-le9-1} again that there exists a constant $\vartheta>0$ such that
$$\mathop{\rm lim\,inf}\limits_{k\in\mathcal{K}}\|\eta^{k1}\|>\vartheta.$$
Note that $\delta_i^k=-\beta_i^k/c_i(x^k)$ for $i\in \mathcal{L}(x^*)$. Then by \eqref{proof-le9-6}, \eqref{proof-le9-3}, \eqref{proof-le9-5} and boundedness of $\{\lambda^k\}$, there is a $\underline{t}_i>0$ such that 

\begin{eqnarray}\label{kexing}
c_i(R_{x^k}(t\eta^k+t^2\tilde{\eta}^k))&=&c_i(x^k)+t\left((\min\{\lambda_i^{k0},0\})^3+\dfrac{\sqrt{2}\beta_i^k\lambda_i^k}{\alpha_i^k}-\theta_k\|\eta^{k1}\|^\nu\right)+o(t)\nonumber\\
&\leq& t\left(\dfrac{\lvert\lambda_i^kc_i(x^k)\lvert}{\alpha_i^k\underline{\mu}}-\bar{\theta}\|\vartheta\|^\nu\right)+o(t)\nonumber\\
&<&0
\end{eqnarray}
for all $k\in\mathcal{K}$ large enough and all $t\in[0,\underline{t}_i]$.

Let $\underline{t}=\min\{\underline{t}_0,\underline{t}_1,\cdots,\underline{t}_{m+\ell}\}$. By \eqref{proof-le9-4} and the search condition \eqref{linef}, we have 
$$F_{\bar{\rho}}(x^{k+1})-F_{\bar{\rho}}(x^k)=F_{\bar{\rho}}(R_{x^k}(t_k\eta^k+t_k^2\tilde{\eta}^k))-F_{\bar{\rho}}(x^k)\leq -\underline{t}\sigma \varsigma\tau\varepsilon$$
for all $k\in\mathcal{K}$ large enough. 
This along with the monotonicity of the sequence $\{F_{\bar{\rho}}(x^k)\}$ shows that $F_{\bar{\rho}}(x^k)\rightarrow-\infty$ as $k\rightarrow \infty$. This contradicts Assumption \ref{A2}, and thus $x^*$ is a KKT point of \eqref{TRpro}.\qed

\section{The Proofs in Section \ref{Sec:superlinear}}\label{AppendixC}
{\bf{Proof of Lemma \ref{Atp}}} (i) If for any $i\in \mathcal{L}(x^*)$, there is no infinite index $\mathcal{K}'\subseteq\mathcal{K}$ such that $\{\mu^k_i\}_{\mathcal{K}'}\rightarrow 0$. Then similar to the second part of the proof of Theorem \ref{th2}, we know that $(x^*,\lambda^*)$ a KKT pair of problem \eqref{TRpro} (i.e. $\hat{\lambda}^*={\lambda}^*$), and there exists an index set $\mathcal{K}'\subseteq\mathcal{K}$ such that $\{\eta^{k0}\}_{\mathcal{K}'}\rightarrow 0_{x^*}$. Since $x^*$ is a KKT point of problem \eqref{Rpro}, we have $c_i(x^*)=0$ for all $i\in\mathcal{E}$. This together with Proposition \ref{tp-p} and the uniqueness of the KKT multiplier, we obtain $\bar{\lambda}^*_\mathcal{I}={\lambda}^*_\mathcal{I}$ and $\bar{\lambda}^*_\mathcal{E}={\lambda}^*_\mathcal{E}-\bar{\rho}e$. It follows from Assumption \ref{A4} that ${\lambda}^*_i>0$ for all $i\in\mathcal{I}(x^*)$. On the other hand, since $\{\eta^{k0}\}_{\mathcal{K}'}\rightarrow 0_{x^*}$, $\{\lambda^{k0}\}_{\mathcal{K}'}\rightarrow \lambda^*\geq 0$ and $\rho_k\equiv\bar{\rho}$ for all $k\in\mathbb{N}$, by the construction of Algorithm \ref{algo1}, we have $\{\lambda^{k0}_\mathcal{E}\}_{\mathcal{K}'}\rightarrow \lambda_\mathcal{E}^*\geq r_2e> 0$. Therefore, the strict complementarity condition of problem \eqref{TRpro} holds at $(x^*,\hat{\lambda}^*)$. Otherwise, there exist some index $i\in \mathcal{L}(x^*)$ and an infinite index set $\mathcal{K}'\subseteq\mathcal{K}$ such that $\{\mu^k_i\}_{\mathcal{K}'}\rightarrow 0$. 
Then, combining with the first part of the proof of Theorem \ref{th2} and following similar reasoning to the above analysis, we can also prove that the strict complementarity condition of problem \eqref{TRpro} holds at $(x^*,\hat{\lambda}^*)$.

(ii) In view of the fact that $x^*\in\mathcal{F}_P$, we obtain $\mathcal{L}(x^*)=\mathcal{I}(x^*)\cup\mathcal{E}$ and ${\rm Hess}_xL_{\bar{\rho}}(x^*,\lambda^*)={\rm Hess}_xL(x^*,\bar{\lambda}^*)$. Then it follows from Assumption \ref{A5} that the SOSC of problem \eqref{TRpro} holds at $(x^*,\hat{\lambda}^*)$. \qed

\vspace{0.5cm}
\noindent
{\bf{Proof of Lemma \ref{le9.4}}} From Theorem \ref{th2}, it is sufficient to show that for any $i\in \mathcal{L}(x^*)$, there does not exist an infinite index set $\mathcal{K}'\subseteq\mathcal{K}$ such that $\{\mu^k_i\}_{\mathcal{K}'}\rightarrow 0$. Suppose by contradiction that there exists some $i\in\mathcal{L}(x^*)$ and an infinite index $\mathcal{K}'\subseteq\mathcal{K}$ such that $\{\mu^k_i\}_{\mathcal{K}'}\rightarrow 0$. Then by the definition of $\mu^k$, we know that $\{\lambda_i^{(k-1)0}\}_{\mathcal{K}'}\rightarrow 0$. This contradicts Lemma \ref{Atp} (i). \qed

\vspace{0.5cm}
\noindent
{\bf{Proof of Proposition \ref{pro4}}}
Since $p^*$ is an isolated accumulation point of $\{p^k\}$, there exists a closed ball $B(p^*,r)=\{p\in X\mid {\rm dist}(p,p^*)\leq r\}$ such that there is no another accumulation point in $B(p^*,r)$. Let 
\begin{equation}\label{proof-pro4-1}
\mathcal{K}_1=\{k\mid {\rm dist}(p^k,p^*)\leq r\},\,\,\mathcal{K}_2=\{k\mid {\rm dist}(p^k,p^*)> r\}.
\end{equation}
We define $\mathcal{K}_3=\{k\mid k\in\mathcal{K}_1,\,k+1\in\mathcal{K}_2\}$. By the definition of $\mathcal{K}_1$ and the fact that $p^*$ is an isolated accumulation point, we have $\{p^k\}_{\mathcal{K}_1}\rightarrow p^*$. Thus there is a positive integer $N_1$ such that 
\begin{equation}\label{proof-pro4-2}
{\rm dist}(p^k,p^*)<\dfrac{r}{4},\,\,\forall k\in\mathcal{K}_1,\,\,k>N_1.
\end{equation}
For a set $S$, we denote the number of elements in set $S$ by $\arrowvert S\arrowvert$. Then we show that $\arrowvert\mathcal{K}_2\arrowvert<\infty$. Suppose by contradiction that $\arrowvert\mathcal{K}_2\arrowvert=\infty$, then $\arrowvert\mathcal{K}_3\arrowvert=\infty$. Since $\mathcal{K}_3\subseteq\mathcal{K}_1$, we obtain $\{p^k\}_{\mathcal{K}_3}\rightarrow p^*$. Therefore, by the assumption of this proposition, there is an infinite index $\mathcal{K}_4\subseteq\mathcal{K}_3$ such that $\{{\rm dist}(p^{k+1},p^k)\}_{\mathcal{K}_4}\rightarrow 0$. Thus there is a positive integer $N_2$ such that 
\begin{equation}\label{proof-pro4-3}
{\rm dist}(p^{k+1},p^k)<\dfrac{r}{4},\,\,\forall k\in\mathcal{K}_4,\,\,k>N_2\,\,{\rm and}\,\,{\rm dist}(p^{k},p^*)>r,\,\,\forall k\in\mathcal{K}_2.
\end{equation}
Let $s\in\mathcal{K}_4$ and $s>\max\{N_1,N_2\}$. It is clear that $s+1\in\mathcal{K}_2$. Then by \eqref{proof-pro4-2} and \eqref{proof-pro4-3}, we have
$${\rm dist}(p^{s+1},p^s)\geq{\rm dist}(p^{s+1},p^*)-{\rm dist}(p^s,p^*)>r-\dfrac{r}{4}=\dfrac{3r}{4}>\dfrac{r}{4},$$
which is a contradiction with \eqref{proof-pro4-3}. Therefore $\arrowvert\mathcal{K}_2\arrowvert<\infty$. Then it follows from \eqref{proof-pro4-1} and $\{p^k\}_{\mathcal{K}_1}\rightarrow p^*$ that $\{p^k\}\rightarrow p^*$, which completes the proof.\qed

\vspace{0.5cm}
\noindent
{\bf{Proof of Corollary \ref{co1}}}
We first show that $\{\eta^{k0}\}\rightarrow 0_{x^*}$. Note that $\{\eta^{k0}\}$ is bounded. Let $\eta^{*0}$ be arbitrary accumulation point of $\{\eta^{k0}\}$. Then there exists an infinite index set $\mathcal{K}\subseteq \mathbb{N}$ such that $\{\eta^{k0}\}_{\mathcal{K}}\rightarrow \eta^{*0}$. It is already shown that $\{\lambda^{k0}\}\rightarrow\lambda^*$. Therefore $(\eta^{*0},\lambda^*)$ must be the unique solution of system \eqref{proof-le11-3}. Note that the unique solution of system \eqref{proof-le11-3} is $(0_{x^*},\lambda^*)$ (see the proof of Theorem \ref{le11}). So we must have $\{\eta^{k0}\}_{\mathcal{K}}\rightarrow 0_{x^*}$. Hence $\{\eta^{k0}\}\rightarrow 0_{x^*}$. Through an analysis similar to the above and using the fact that  $\left\{\left(\min\{\lambda_i^{k0},0\}\right)^3\right\}\rightarrow 0$, we can obtain  $\{\eta^{k1}\}\rightarrow 0_{x^*}$. Then it follows from \eqref{relationship30} that $\{\eta^{k}\}\rightarrow 0_{x^*}$. Thus, part (i) holds. By \eqref{relationship20} and \eqref{relationship30}, part (ii) is also true.  Part (iii) follows from the definition of $\mu^k$ and Assumption \ref{A6}. Finally, the set $\mathcal{L}_k$ detects $\mathcal{L}(x^*)$ completely whenever $(x^k,\lambda^k)$ is close to $(x^*,\lambda^*)$ under Assumption \ref{A4}.\qed

\vspace{0.5cm}
\noindent
{\bf{Proof of Lemma \ref{le12}}}
It follows from \eqref{alpha}, Corollary \ref{co1} and Lemma \ref{Atp} (i) that $\{\alpha_i^k\}\rightarrow 1$ and $\{\beta_i^k\}\rightarrow 0$ for $i\in \mathcal{L}(x^*)$.
On the other hand, note that
$$(\lambda_i^k\delta^k_i)^2=\dfrac{(\lambda_i^k)^2}{\sqrt{c_i^2(x^k)+(\mu_i^k)^2}\left(\sqrt{c_i^2(x^k)+(\mu_i^k)^2}+\mu_i^k\right)}.$$
For $i\in \mathcal{L}(x^*)$, it follows from Corollary \ref{co1} and Lemma \ref{Atp} (i) that $\{(\lambda_i^k\delta^k_i)^2\}\rightarrow 1/2$. Thus part (i) is established. 
Part (ii) also holds from Corollary \ref{co1} and Lemma \ref{Atp} (i).\qed

\vspace{0.5cm}
\noindent
{\bf{Proof of Lemma \ref{le13}}}
It follows from Corollary \ref{co1} that $\mathcal{L}_k=\mathcal{L}(x^*)$ for $k$ large enough. Thus the direction $\tilde{\eta}^k$ is computed by attempting to solve \eqref{subpro} in $\eta$ with $\mathcal{L}_k=\mathcal{L}(x^*)$, which is equivalent to solving the following linear system in ($\eta, \lambda$)	
\begin{equation}\label{proof-le13-0}
\begin{aligned}
	&\mathcal{H}_k\eta+\sum_{i\in \mathcal{L}(x^*)}\lambda_i{\rm grad}c_i(x^k)=0_{x^k},\\
	&\langle{\rm grad}c_i(x^k),\eta\rangle=-\varpi_k-c_i(R_{x^k}(\eta^k)),\quad \forall i\in \mathcal{L}(x^*).
\end{aligned}
\end{equation}
Without loss of generality, we assume that $\mathcal{L}(x^*)=\{1,2,\cdots,\arrowvert \mathcal{L}(x^*) \arrowvert\}$, where $\arrowvert \mathcal{L}(x^*) \arrowvert$ denotes the number of elements of the set $\mathcal{L}(x^*)$. Define $\mathcal{B}_k:T_{x^k}\mathcal{M}\times\mathbb{R}^{\arrowvert \mathcal{L}(x^*) \arrowvert}\rightarrow T_{x^k}\mathcal{M}\times\mathbb{R}^{\arrowvert\mathcal{L}(x^*) \arrowvert}$ as
\begin{equation}\label{proof-le13-01}
\mathcal{B}_k(\eta,\lambda)
=\left(\mathcal{H}_k[\eta]+\sum_{i=1}^{\arrowvert \mathcal{L}(x^*) \arrowvert}\lambda_i{\rm grad}c_i(x^k),\Lambda'_k\right),
\end{equation}
where $$\Lambda'_k=\left(\langle{\rm grad}c_1(x^k),\eta\rangle,\cdots,\langle{\rm grad}c_{\arrowvert \mathcal{L}(x^*) \arrowvert}(x^k),\eta\rangle\right)^{\top}\in\mathbb{R}^{\arrowvert \mathcal{L}(x^*) \arrowvert}.$$
It is clear that $\mathcal{B}_k$ is a continuous linear operator. On the other hand, by Assumptions \ref{A1} and \ref{A3}, we know that $\mathcal{B}_k$ is nonsingular for all $k$ large enough. Furthermore, similar to the proof in \cite[Lem. 3.1]{qi2000new}, we can show that the sequence $\{\|\mathcal{B}_k^{-1}\|_{\rm op}\}$ (with $k$ being large enough) is bounded.

Using the Taylor expansion for $c_i$, it follows that
\begin{equation}\label{proof-le13-1}
-c_i(R_{x^k}(\eta^k))=-c_i(x^k)-\langle{\rm grad}c_i(x^k),\eta^k\rangle+O(\|\eta^k\|^2).
\end{equation}
By the definition of $\delta^k_i$, we obtain $c_i(x^k)=-\beta^k_i/\delta^k_i$. On the other hand, for $i\in \mathcal{L}(x^*)$ and $k$ large enough, from Lemma \ref{le12} (i), we have $\alpha_i^k\neq 0$. It then follows from \eqref{s42}, \eqref{eta}, $\nu>2$ and $\lambda^{k0}\rightarrow\lambda^*$ that $$\dfrac{\sqrt{2}\beta^k_i\lambda^k_i}{\alpha_i^k}=\langle{\rm grad}c_i(x^k),\eta^k\rangle +o(\|\eta^k\|^2).$$ 
Thus by Lemma \ref{le12} and \eqref{proof-le13-1}, we have
\begin{equation*}\label{proof-le13-2}
-\varpi_k-c_i(R_{x^k}(\eta^k))=-\varpi_k+\left(\dfrac{\alpha_i^k}{\sqrt{2}\delta^k_i\lambda_i^k}-1\right)\langle{\rm grad}c_i(x^k),\eta^k\rangle+O(\|\eta^k\|^2).
\end{equation*}
The definition of $\varpi_k$ and Lemma \ref{le12} show that $\varpi_k=o(\|\eta^k\|^2)$. Then by \eqref{proof-le13-0}, \eqref{proof-le13-01} and the boundedness of $\{\|\mathcal{B}_k^{-1}\|\}$, we complete the proof of this lemma.\qed

\end{document}